\documentclass[11pt,a4paper,final]{article}

\usepackage{amsmath}       
\usepackage{amsfonts}      
\usepackage{amssymb}       
\usepackage{amsthm}        
\usepackage{mathrsfs}      
\usepackage{graphicx}      
\usepackage{calc}          
\usepackage[rm]{titlesec}  

\renewcommand{\baselinestretch}{1.25}
\newlength{\vslength}
\newcommand{\ie}{{\it i.e.}}
\newcommand{\cf}{{\it c.f.}}
\newcommand{\eg}{{\it e.g.}}

\newcommand{\lhs}{{\it l.h.s.}}
\newcommand{\rhs}{{\it r.h.s.}}
\newcommand{\iid}{{\it i.i.d.}}

\newcommand{\RR}{{\mathbb R}}

\newcommand{\scrA}{{\mathscr A}}
\newcommand{\scrB}{{\mathscr B}}

\newcommand{\scrD}{{\mathscr D}}

\newcommand{\scrP}{{\mathscr P}}

\newcommand{\scrT}{{\mathscr T}}
\newcommand{\scrX}{{\mathscr X}}

\newcommand{\al}{{\alpha}}
\newcommand{\ep}{{\epsilon}}

\newcommand{\tht}{{\theta}}

\newcommand{\Tht}{{\Theta}}

\DeclareMathOperator*{\winf}{{\phantom{p}inf\phantom{p}}}
\newcommand{\vinf}{\winf\limits}

\newcommand{\ft}[2]{{\textstyle{\frac{#1}{#2}}}}
\newcommand{\co}[1]{{\mathop{\mathrm{co}}(#1)}}

\newcommand{\conv}[1]%
  {{\mathrel{\,\xrightarrow{\widthof{\,#1\,}}\,}}}
\newcommand{\convas}[1]%
  {{\mathrel{\,\xrightarrow{\widthof{\,#1\text{-a.s.}\,}}\,}}}
\newcommand{\convprob}[1]%
  {{\mathrel{\,\xrightarrow{\widthof{\,#1\,}}\,}}}
\newcommand{\convweak}[1]%
  {{\mathrel{\,\xrightarrow{\widthof{\,#1\text{-w.}\,}}\,}}}

\newcommand{\twobytwo}[4]%
  {\left(\begin{array}{cc} #1 & #2 \\ #3 & #4 \end{array}\right)}
\newcommand{\twovec}[2]%
  {\left({\begin{array}{c} #1\\#2 \end{array}}\right)}

\numberwithin{equation}{section}
\theoremstyle{plain}

\newtheorem{theorem}{Theorem}[section]

\newtheorem{lemma}[theorem]{Lemma}

\newtheorem{proposition}[theorem]{Proposition}

\newtheorem{corollary}[theorem]{Corollary}

\newtheorem{conjecture}[theorem]{Conjecture}

\theoremstyle{remark}
\newtheorem{remark}[theorem]{Remark}

\newtheorem{example}[theorem]{Example}

\renewenvironment{proof}{\noindent{\bf proof}\;\;}{\hfill$\Box$\par}

\begin{document}

\thispagestyle{empty}

\title{\vspace*{-9mm}\huge\rm 
  Criteria for posterior consistency\\and convergence at a rate}
\author{\sc B. J. K.~{Kleijn}$^{1}$ and
  Y. Y.~{Zhao}$^{2}$\footnote{Corresponding author: {\tt
      {yanyunzhao@znufe.edu.cn}} or {\tt {yyunzhao@gmail.com}}}\\
  {\small\slshape ${}^{1}$ Korteweg-de~Vries Institute for Mathematics,
    University of Amsterdam}\\
  {\small\slshape ${}^{2}$ Wenlan School of Business,
    Zhongnan University of Economics and Law}\\[2mm]
  }
\date{\sc May 2017}
\maketitle

\begin{abstract}
Frequentist conditions for asymptotic suitability of
Bayesian procedures focus on lower bounds for prior
mass in Kullback-Leibler neighbourhoods of the data
distribution. The goal of this paper is to investigate
the flexibility in criteria for posterior consistency
with \iid\ data. We formulate a versatile posterior
consistency theorem that applies both to well- and
mis-specified models and which we use to re-derive
Schwartz's theorem, consider Kullback-Leibler
consistency and formulate consistency theorems in
which priors charge metric balls. It is generalized
to sieved models with Barron's negligible prior mass
condition and to separable models with variations on
Walker's consistency theorem. Results also apply to
marginal semi-parametric consistency: support boundary
estimation is considered explicitly and consistency is
proved in a model for which Kullback-Leibler priors
do not exist. Other examples
include consistent density estimation in
mixture models with Dirichlet or Gibbs-type priors of
full weak support. Regarding posterior convergence at
a rate, it is shown that under a mild integrability
condition, the second-order Ghosal-Ghosh-van~der~Vaart
prior mass condition can be relaxed to a lower bound
to the prior mass in Schwartz's Kullback-Leibler
neighbourhoods. The posterior rate of convergence is
derived in a simple, parametric model for heavy-tailed
distributions in which the Ghosal-Ghosh-van~der~Vaart
condition cannot be satisfied by any prior.
\end{abstract}


\section{Introduction and main result}
\label{sec:intro}

Aside from computational issues, the most restrictive aspects of
non-parametric Bayesian methods result from limited availability of
priors. In general, distributions on infinite dimensional spaces are
relatively hard to define and control technically, so unnecessary
elimination of candidate priors is highly undesirable. Specifying
to frequentist asymptotic aspects, the \emph{conditions} that Bayesian
limit theorems pose on priors play a crucial role: it is the goal of
this paper to extend the range of criteria on the prior for posterior
consistency \cite{Ghosh99} and convergence at a rate \cite{Ghosal00},
showing asymptotic suitability for a wider range of priors. We accept
that this may go at the expense of additional model conditions.

\subsection{Introduction}

As early as the 1940's, J.~Doob \cite{Doob48} studied posterior
limits as a part of his exploits in martingale convergence: if
the data forms an infinite \iid\ sample from a distribution
$P_{\tht_0}$ on a measurable space $(\scrX,\scrA)$ in a model
$\scrP=\{P_\tht:\tht\in\Tht\}$ where $\Tht$ and sample space are
Polish spaces and $\Tht\rightarrow\scrP:\tht\mapsto P_\tht$ is
one-to-one, then for any prior $\Pi$ on $\Tht$ the posterior is
consistent, $\Pi$-almost-surely. Notwithstanding its remarkable
generality and its Bayesian interpretation,
Doob's theorem is not quite satisfactory to the frequentist
interested in non-parametric models, in that the null-set of
the prior on which inconsistency may occur can be very large,
as was stressed by Schwartz \cite{Schwartz61} and amplified
repeatedly by Freedman \cite{Freedman63}.

To frequentists Freedman's counterexamples discredited
Bayesian methods for non-parametric statistics greatly. The resulting
under-appreciation was hard to justify, given that a frequentist
alternative to Doob's theorem had existed since 1965: Schwartz's
consistency theorem \cite{Schwartz65} below concerns models $\scrP$
that are dominated by a $\sigma$-finite measure $\mu$ (with densities
$p=dP/d\mu$ for $P\in\scrP$) and departs from,
\begin{equation}
  \label{eq:posterior}
  \Pi(\,A\,|\,X_1,\ldots,X_n\,)= 
  {\displaystyle{\int_A\prod_{i=1}^n p(X_i)\,d\Pi(P)}} \biggm/
  {\displaystyle{\int_\scrP\prod_{i=1}^n p(X_i)\,d\Pi(P)}},
\end{equation}
the standard expression for the posterior in dominated models.
Note that Schwartz's
own formulation of the theorem assumes the existence of certain
test sequences, which exist without further assumptions if one is
interested in weak consistency, and can be constructed for
Hellinger consistency if the model is totally bounded with respect to
the Hellinger metric $H$.
The formulation chosen below focuses on the latter case. (Here
and below, let $N(\ep,\scrP,H)$ denote the radius-$\ep$ Hellinger
covering number $\scrP$. Hellinger totally boundedness of $\scrP$
corresponds to finiteness of $N(\ep,\scrP,H)$ for all $\ep>0$.)
\begin{theorem}
\label{thm:schwartz}
(Schwartz (1965))\\
Let the model $\scrP$ be totally bounded relative to the Hellinger
metric $H$ and let $X_1, X_2, \ldots$ be $\iid-P_0$ for some
$P_0\in\scrP$. If $\Pi$ is a Kullback-Leibler prior, \ie\ for all
$\delta>0$,  
\begin{equation}
  \label{eq:KLprior}
  \Pi\Bigl(\,P\in\scrP\,:\,
    -P_0\log\frac{dP}{dP_0}<\delta\,\Bigr) > 0,
\end{equation}
then the posterior is Hellinger consistent at $P_0$, that is,
\begin{equation}
  \label{eq:Schwartzconsistency}
  \Pi\bigl(\,P\in\scrP\,:\,H(P,P_0)>\ep\,\bigm|\,X_1,\ldots,X_n\,\bigr)
  \convas{P_0}0,
\end{equation}
for every $\ep>0$.
\end{theorem}
This type of formulation involves a specific underlying distribution
$P_0$ and thus avoids the possibility of non-validity on null-sets of
the prior. Schwartz's theorem does not cover all examples, however.
\begin{example}
\label{ex:noKLpriors}
Consider an \iid\ sample $X_1,X_2,\ldots$ from a distribution $P_0$
with Lebesgue density $p_0:\RR\rightarrow\RR$ that is supported on
an interval of known width (say, $1$)  but unknown location. The
model is parametrized in terms of a density $\eta$ supported on
$[0,1]$ with $\eta(x)>0$ for all $x\in[0,1]$ and a location
$\tht\in\RR$:
\[
  p_{\tht,\eta}(x) = \eta(x-\tht)\,1_{ [\tht,\tht+1] }(x).
\]
Note that if $\tht$ does not equal $\tht'$,
\[
  -P_{\tht,\eta}\log\frac{p_{\tht',\eta'}}{p_{\tht,\eta}} = \infty
\]
for all $\eta,\eta'$. Therefore Kullback-Leibler neighbourhoods
do not have any extent in the $\tht$-direction and no prior can be
a Kullback-Leibler prior in this model. (See
example~\ref{ex:fixedwidthdomain} for more.) 
\end{example}

Totally-boundedness of the model is a restrictive condition: in the
case of Schwartz's theorem that condition can be mitigated in several
distinct ways, for example by use of the so-called \emph{Le~Cam-dimension}
of the model \cite{LeCam73}. An extension for
non-totally-bounded models of a more Bayesian flavour is due to Barron
(see, for example, \cite{Barron99} and section~4.4.2 of \cite{Ghosh03}),
who demonstrates posterior consistency for Kullback-Leibler priors,
based on a partition of the model into a subset of bounded
Hellinger metric entropy and a subset of negligibly small prior mass.
More recently, Walker has proposed a method that does not
depend on finite covers but strengthens condition (\ref{eq:KLprior})
with a summability condition \cite{Walker04}. (For more, see
subsection~\ref{sub:walker}.)
\begin{theorem}
\label{thm:walker}
(Walker (2004))\\
Let the model $\scrP$ be Hellinger separable and let $X_1, X_2, \ldots$
be $\iid-P_0$ for some $P_0\in\scrP$. Let $\ep>0$ be given and let
$\{V_i:i\geq1\}$ be a countable cover of $\scrP$ by balls of a radius
$0<\delta<\ep$. If $\Pi$ is a Kullback-Leibler prior and in addition,
\begin{equation}
  \label{eq:walker}
  \sum_{i\geq1} \Pi(V_i)^{1/2} < \infty,
\end{equation}
then $\Pi(\,P\in\scrP\,:\,H(P,P_0)>\ep\,|\,X_1,\ldots,X_n\,)\convas{P_0}0$.
\end{theorem}
It appears that, thus far, no clear relationship between Schwartz's
and Walker's theorems has been established. Particularly, while
Schwartz's theorem poses only a \emph{lower bound} for prior mass
(around $P_0$), Walker's theorem also requires an \emph{upper bound}
on prior mass (further away from $P_0$), suggesting that
theorems~\ref{thm:schwartz} and~\ref{thm:walker} differ materially
rather than superficially.

Another significant extension of the theory on posterior convergence
is formed by results concerning posterior convergence \emph{at a rate}.
Extension of Schwartz's theorem to posterior rates of convergence
\cite{Ghosal00,Shen01} applies Barron's sieve idea and a more intricate
minimax argument \cite{Birge83,Birge84} to a shrinking sequence of
Hellinger neighbourhoods and employs a more specific, rate-related
version of the Kullback-Leibler condition (\ref{eq:KLprior}) for
the prior. The preferred formulation takes the following form.
\begin{theorem}
\label{thm:GGV} {\it (Ghosal, Ghosh and van der Vaart, 2000)}\\
Let $X_1, X_2, \ldots$ be $\iid-P_0$ for some $P_0$ in the model
$\scrP$ which we endow with the Hellinger metric $H$. Let
$(\ep_n)$ be a sequence with $\ep_n\downarrow 0$ and
$n\ep_n^2\rightarrow\infty$. Let $C>0$ and measurable
$\scrP_n\subset\scrP$ be such that, for large enough $n$,
\begin{itemize}
\item[(i)] the $\scrP_n$ are of bounded $H$-entropy:
  $N(\ep_n,\scrP_n,H)\leq e^{n\ep_n^2}$;
\item[(ii)] the prior mass outside $\scrP_n$ is bounded:
  $\Pi(\scrP\setminus\scrP_n)\leq e^{-n\ep_n^2(C+4)}$;
\item[(iii)] the prior $\Pi$ is such that,
  \begin{equation}
    \label{eq:GGV}
    \Pi\Bigl(\,P\in\scrP\,:\,-P_0\log\frac{dP}{dP_0}<\ep_n^2,\,
      P_0\Bigl( \log\frac{dP}{dP_0} \Bigr)^2<\ep_n^2\, \Bigr)\geq e^{-Cn\ep_n^2}.
  \end{equation}
\end{itemize}
Then the posterior converges in Hellinger distance at rate $\ep_n$, \ie
\[
  \Pi\bigl(\,P\in\scrP:\,H(P,P_0)>M\ep_n\,\bigm|\,X_1,\ldots,X_n\,\bigr)
  \convprob{P_0}0,
\]
for all $M>0$ that are sufficiently large.
\end{theorem}

Aside from examples like \ref{ex:noKLpriors}, there are straightforward
circumstances in which condition (\ref{eq:GGV}) cannot be satisfied
by any prior. In the example below, heavy-tailed distributions are
found for which integrability of squared log-density ratios is
violated.
\begin{example}
\label{ex:noGGVpriors}
Consider an \iid\ sample of integers $X_1, X_2,\ldots$ from a
distribution $P_a$, $(a\geq1)$, defined by,
\begin{equation}
  \label{eq:heavytail}
  p_a(k) = P_a(X=k) = \frac{1}{Z_a}\frac{1}{k^a(\log k)^3}
\end{equation}
for all $k\geq2$, with $Z_a=\sum_{k\geq2} k^{-a}(\log k)^{-3}<\infty$.
As it turns out, for $a=1$, $b>1$,
\[
  -P_a\log\frac{p_b}{p_a}<\infty, \quad
  P_a\Bigl(\log\frac{p_b}{p_a}\Bigr)^2=\infty.
\]
Therefore, Schwartz's KL-condition (\ref{eq:KLprior}) for the prior
for the parameter $a$ can be satisfied but there exists no prior such
that (\ref{eq:GGV}) is satisfied for all $P_0$ in the model. (See
example~\ref{ex:heavytail} for more.) 

In fact,
if we change the third power of the $\log$-factor in the denominator of
(\ref{eq:heavytail}) to a square, Schwartz's KL-priors also do not
exist. The above construction is indicative of a more general
problem: for any $P_0$ it is possible to find distributions $P$
with densities $p$ that are `wild enough' to cause log-likelihood
ratios $\log p/p_0$ to loose integrability or square-integrability.
Instances of posterior inconsistency \cite{Barron99,Diaconis86} and
the phenomenon of \emph{data-tracking} \cite{Walker05} sketch a
similar qualitative picture of situations where posterior consistency
fails.
\end{example}

Schwartz's theorem and its rate-specific version have become the
standard frequentist tools for the asymptotic analysis of Bayesian
posteriors, almost to the point of exclusivity. As a consequence, lower
bounds for prior mass in Kullback-Leibler neighbourhoods \cf\
(\ref{eq:KLprior}) and (\ref{eq:GGV}) are virtually the {\em only}
criteria frequentists apply to priors in non-parametric
asymptotic analyses (notable exceptions are made in the
examples of \cite{Castillo14,Castillo15,Hoffman15}; see,
however, lemma~\ref{lem:KLequiv} below). Since these
lower bounds on prior weights of Kullback-Leibler-neighbourhoods
are sufficient conditions applicable for \iid\ data, it is not clear
if other criteria for the prior can be formulated.
The goal of this paper is to investigate whether flexibility can be
gained with regard to criteria for prior choice, with the ultimate goal
of formulating new consistency theorems based on a greater variety of
suitability conditions for priors. The goal is \emph{not} to generalize
conditions of Schwartz's theorem or to sharpen its assertion; rather
we want to show that stringency with regard to the prior can be relaxed
at the expense of stringency with regard to conditions on the model.

\subsection{Main result}

The main result is summarized in the next theorem: we have in
mind a fixed model subset $V$ for which we want to demonstrate
asymptotically vanishing posterior mass. Although suitable also for
hypothesis testing in principle, our main interest lies with the
situation where $V$ is the complement of an open neighbourhood of
$P_0$. Following the ideas of \cite{Schwartz65,LeCam73,Birge83,Birge84}
the set $V$ is covered by a finite collection of subsets $V_1,\ldots,V_N$
to be tested against $P_0$ separately with the help of the minimax
theorem. However, here, we involve the prior in the minimax problem
from the start: each $V_i$ is matched with a model subset
$B_i$ (which can be thought of as a `neighbourhood' of $P_0$ if the
model is well-specified) such that $\Pi(B_i)>0$ and inequality
(\ref{eq:power}) below is satisfied. It will be shown
that the $B_i$ can often be chosen as Kullback-Leibler neighbourhoods
(as in Schwartz's theorem), but alternative choices for the $B_i$ become
possible as well.

Throughout this paper and in the formulation below, we assume that
the model is dominated and we use posterior (\ref{eq:posterior}).
Let $\co{V}$ denote the convex hull of $V$ and let $P_n^\Pi$ ($n\geq1$),
denote the $n$-fold prior predictive distributions:
$P_n^\Pi(A)=\int P^n(A)\,d\Pi(P)$, for all $A\in\sigma(X_1,\ldots,X_n)$.
Furthermore, for given $\al\in[0,1]$, model subsets $B,W$ and a given
distribution $P_0$, define,
\begin{equation}
  \label{eq:testingpower}
  \pi_{P_0}(W,B;\al)
    = \sup_{P\in W} \, \sup_{Q\in B} \, P_0\Bigl(\frac{dP}{dQ}\Bigr)^\al,
\end{equation}
related to the Hellinger transform (see appendix~\ref{app:hell}).
Also define $\pi_{P_0}(W,B)$ to be equal to $\inf_{\al\in[0,1]}\pi_{P_0}(W,B;\al)$.
\begin{theorem}
\label{thm:main}
Let the model $\scrP$ be given and let $X_1, X_2, \ldots$ be
\iid-$P_0$ distributed. Assume that $P_0^n\ll P_n^\Pi$ for all $n\geq1$.
For some $N\geq1$ let $V_1,\ldots,V_N$ be measurable model subsets. If there
exist measurable model subsets $B_1,\ldots, B_N$ such that for every
$1\leq i\leq N$,
\begin{equation}
  \label{eq:power}
      \pi_{P_0}(\,\co{V_i},\,B_i\,) < 1,
\end{equation}
$\Pi(B_i)>0$ and $\sup_{Q\in B_i} P_0(dP/dQ)<\infty$ for all $P\in V_i$,
then,
\begin{equation}
  \label{eq:consistency}
  \Pi(\,V\,|\,X_1,\ldots,X_n\,)\convas{P_0}0,
\end{equation}
for any $V\subset\bigcup_{1\leq i\leq N} V_i$.
\end{theorem}
Although this angle will not be pursued further in this paper, it is noted
that $P_0$ is \emph{not required} to be in the model $\scrP$ so that the
theorem applies both to well- and to mis-specified models \cite{Kleijn06}
in the form stated. Furthermore, in subsection~\ref{sub:KLpriors} it is
shown that condition (\ref{eq:power}) is \emph{equivalent} in quite some
generality to separation of $B_i$ and $\co{V_i}$ in Kullback-Leibler
divergence with respect to $P_0$,
\begin{equation}
  \sup_{Q\in B_i}-P_0\log\frac{dQ}{dP_0}
    < \vinf_{\phantom{|}P\in \co{V_i}}-P_0\log\frac{dP}{dP_0},
\end{equation}
underlining the fundamental nature of condition (\ref{eq:KLprior}). 
But even if we keep this equivalence in mind, it may be possible
to formulate less demanding criteria for the choice of the prior at 
the expense of more stringent model conditions: the theorem
is uncommitted regarding the nature of the $V_i$, and, more
importantly, we may use any $B_i$ that
\emph{(i)} allow uniform control of $P_0(p/q)^\al$, and \emph{(ii)} allow
convenient choice of a prior such that $\Pi(B_i)>0$. The two requirements
on $B_i$ leave room for trade-offs between being `small enough' to
satisfy \emph{(i)}, but `large enough' to enable a choice for $\Pi$
that leads to \emph{(ii)}. The freedom to choose $B$'s and $\Pi$ lends
the method the desired flexibility: given $\scrP$ and $V$, can we
find $V_i$'s, $B_i$'s and a prior $\Pi$ like above?

In what follows it is shown that Schwartz's theorem, Barron's sieve
generalization, Walker's theorem and posterior rates of convergence
\cf\ Ghosal-Ghosh-van~der~Vaart can all be related to
theorem~\ref{thm:main}. In section~\ref{sec:cons}, the denominator of
expression (\ref{eq:posterior}) is considered in detail and
theorem~\ref{thm:main} is proved. In section~\ref{sec:variations} we
establish that condition (\ref{eq:power}) is equivalent to
KL-separation. Based on that,
Schwartz's theorem is re-derived and several variations are considered,
\eg\ posterior consistency in Kullback-Leibler divergence with a prior
satisfying (\ref{eq:KLprior}) and Hellinger consistency with priors that
charge metric balls. In section~\ref{sec:sep} it is shown that the
totally-boundedness condition is not essential, in two distinct ways:
firstly we give a version of the theorem involving a sieve of
submodels with finite covers, whose complements satisfy Barron's
negligible prior mass condition. Secondly, we consider variations on
Walker's theorem to guarantee Hellinger consistency with
Kullback-Leibler priors that satisfy certain summability conditions.
In section~\ref{sec:rates} we consider posterior rates of convergence
and show that the second-order KL-condition on the prior of
(\ref{eq:GGV}) can be replaced by a rate-specific version of
Schwartz's KL-condition (\ref{eq:KLprior}): for some $K>0$,
\[
  \Pi\Bigl(\,P\in\scrP\,:\,-P_0\log\frac{dP}{dP_0}<\ep_n^2\,\Bigr)
    \geq e^{-K n\ep_n^2},
\]
under a mild integrability condition on the model.

To apply the results and demonstrate that proposed methods allow
for considerable flexibility, section~\ref{sec:boundary} concerns
semi-parametric estimation of support boundary points for a density
on a bounded interval in $\RR$ \cite{Kleijn15}. The last section
contains a short discussion on applications, including consistency
in non-parametric density estimation with various Dirichlet mixtures,
and the difficult examples~\ref{ex:noKLpriors} and~\ref{ex:noGGVpriors}.
We conclude with two appendices, one on the Hellinger
transform and another containing proofs.

\subsection*{Two notes on supports}
Below, the focus is on expectations of the form $P_0(p/q)^\al$
where $p$ and $q$ are probability densities and $P_0$ is the
marginal for the \iid\ sample. Because the proof of lemma~\ref{lem:cons}
is in $P_0$-expectation, an indicator $1_{\{p_0>0\}}(x)$ is implicit
in all calculations that follow. Because of (\ref{eq:posterior})
and because we look at moments of $p/q$, an indicator $1_{\{p>0\}}(x)$
can also be thought of as a factor in the integrand. Because we require
finiteness of $P_0(p/q)$, $q>0$ is implicit whenever $p_0>0$ and
$p>0$, so in expressions of this form an indicator $1_{\{q>0\}}(x)$
may also be thought of as implicit. Secondly, to avoid confusion, we
say that \emph{$P$ lies in the support of a measure $\Pi$} if $\Pi(U)>0$
for all neighbourhoods $U$ of $P$.


\section{Posterior consistency}
\label{sec:cons}

To establish the basics, the model $(\scrP,\scrB)$ is a
measurable space consisting of Markov kernels $P$ on a sample space
$(\scrX,\scrA)$: the map $A\mapsto P(A)$ is a probability measure for
every $P\in\scrP$ and the map $P\mapsto P(A)$ is measurable for
every $A\in\scrA$. Assuming the model is dominated by a $\sigma$-finite
measure (with density $p$ for $P\in\scrP$), a \emph{prior} probability
measure $\Pi$ on $(\scrP,\scrB)$ gives rise
to the \emph{posterior} \cf\ (\ref{eq:posterior}), which is a Markov
kernel from $(\scrX^n,\scrA^n)$ into $(\scrP,\scrB)$. We
take the frequentist \iid\ perspective, \ie\ we assume that there exists
a distribution $P_0$ on $(\scrX,\scrA)$ such that 
$(X_1,\ldots, X_n)\sim P_0^n$. As a consequence 
expression~(\ref{eq:posterior}) does not make sense automatically:
for the denominator to be non-zero with $P_0^n$-probability one, we
impose that,
\begin{equation}
\label{eq:dompriorpred}
  P_0^n \ll P_{n}^{\Pi},
\end{equation}
for every $n\geq1$, where $P_n^{\Pi}$ is the prior predictive
distribution. If (\ref{eq:dompriorpred}) is not satisfied, it is
possible that expression (\ref{eq:posterior}) for the posterior is
ill-defined for infinitely many $n\geq1$ with $P_0^\infty$-probability
one. The following lemma provides a sufficient condition for
(\ref{eq:dompriorpred}).
\begin{proposition}
\label{prop:dompriorpred}
If $P_0$ lies in the Hellinger support of the prior $\Pi$, then
$P_0^n \ll P_{n}^{\Pi}$, for all $n\geq1$.
\end{proposition}
Another way to satisfy (\ref{eq:dompriorpred}) arises as an implication
of Barron's notion of \emph{matching} \cite{Barron86}: given (a sequence of
dominating measures $(\mu_n)$ and) a sequence of $\mu_n$-probability
densities $(f_n)$, another such sequence $(g_n)$ is said to
\emph{match} $(f_n)$, if there exists a constant $c>0$ such that for all
$n$ large enough,
\begin{equation}
  \label{eq:matching}
  e^{-nc}\,g_n(X_1,\ldots,X_n) \leq f_n(X_1,\ldots,X_n)
    \leq e^{nc}\,g_n(X_1,\ldots,X_n),
\end{equation}
almost-surely, for $(X_1,\ldots,X_n)$ distributed \cf\ the density $f_n$.
Following Barron we associate $f_n$ and $g_n$ with (densities for)
$P_0^n$ and $P^{\Pi}_n$, noting that \emph{matching} for the two implies that
domination condition (\ref{eq:dompriorpred}) is satisfied. Matching of
$P^{\Pi}_n$ with $P_0^n$ also arises as the
central lower-bound for the denominator of the posterior in proofs of
Schwartz's theorem (see \eg\ inequalities~(6) in \cite{Schwartz65}),
so the following corollary does not come as a great surprise.
\begin{corollary}
\label{cor:KLpriorpred}
If $\Pi$ is a KL prior, then $P_0^n\ll P_{n}^{\Pi}$ for
all $n\geq1$.
\end{corollary}
So under Schwartz's prior mass condition, one does not worry about
condition~(\ref{eq:dompriorpred}); it plays a role only if one is
interested in priors that are not Kullback-Leibler priors, like
that of example~\ref{ex:fixedwidthdomain}.
\begin{example} To illustrate the denominator problem by example,
consider the following semi-parametric regression problem: one observes
pairs $(X_i,Y_i)$, $i\geq1$, of real-valued random variables related
through $Y=f(X)+e$ for some regression function $f:\RR\rightarrow\RR$.
Assume for simplicity that $f\geq0$ and that the distribution of
the co-variate $X$ is such that for all $\delta>0$, $P(f(X)<\delta)>0$.
The errors $e_1,e_2,\ldots$ are independent of $X$ and \iid\ with a
distribution supported on $[\tht,\infty)$, for some $\tht\in\RR$.
The problem occurs when the statistician believes that his errors
are \emph{positive} with probability one, while their true distribution
assigns (small but) non-zero probability to \emph{negative} outcomes.
(In finance examples of this type abound, arising when one
anticipates lower-bounded returns (for example a hedged return,
the total return on a bond or an auction price) from an incomplete
or simplified model for downside risk.)

The statistician will make a choice for the prior $\Pi$ that reflects
his belief and not place mass around negative values of the parameter
$\tht$. When the experiment is conducted, sooner or later a negative
value of the error will occur in conjunction with a small value of
$f(X)$, resulting in a negative value for $Y$ that is impossible
according to the part of the model that receives any prior mass.
Consequently, the likelihood evaluates to zero $\Pi$-almost-everywhere
in the model, resulting in a posterior that is ill-defined. Clearly,
$\Pi$ does not satisfy (\ref{eq:KLprior}) and the support mismatch
shows that $P_0$ does not lie in the Hellinger support of $\Pi$ either.
\end{example}

\subsection{A sketch of the proof of theorem~\ref{thm:main}}

To prove consistency with a given prior, one tries to show that the
posterior concentrates all its mass in neighbourhoods of $P_0$
asymptotically, often metric balls centred on $P_0$.
The first lemma in this section asserts that, under the condition
that specific test-sequences for covers of the complement exist,
posterior concentration follows. The proof is inspired by
that of Schwartz's theorem \cite{Schwartz65,Ghosal00,Ghosh03} and
Le~Cam's \emph{dimensionality restrictions} \cite{LeCam73}. Central
is the existence of certain test sequences, in a construction 
Birg\'e refers to as {\it covering-a-ball-by-smaller-balls}
\cite{Birge83,Birge84}, which has its roots in \cite{LeCam73}.
The argument is essentially an application of
the minimax theorem (see, for example, section~16.4 of \cite{LeCam86},
or section~45 of \cite{Strasser85}): the specific form it takes
in this paper is an adaptation of methods developed in \cite{Kleijn06}.
The essential difference between lemma~\ref{lem:cons} and existing
Bayesian limit theorems is that posterior numerator and
denominator are dealt with simultaneously rather than separately.
As a result the prior $\Pi$ is one of the factors that determines 
testing power and can be balanced against model properties directly.

In the following lemma $V$ is a fixed set (\eg\ the complement of
an open neighbourhood of $P_0$) for which we
want to prove asymptotically vanishing posterior mass.
We cover $V$ by a finite number of model subsets $V_1,\ldots,V_N$
such that for each $V_i$, a special type of test sequence exists.
In the next subsection, we give conditions for the existence of
such sequences.
\begin{lemma}
\label{lem:cons}
Assume that $P_0^n\ll P_n^\Pi$ for all
$n\geq1$. For some $N\geq1$, let $V_1,\ldots,V_N$ be a finite
collection of measurable model subsets. If there exist constants
$D_i>0$ and test sequences $(\phi_{i,n})$ for all $1\leq i\leq N$
such that,
\begin{equation}
  \label{eq:testcond}
  P_0^n\phi_{i,n} +
  \sup_{P\in V_i} P_0^n \frac{dP^n}{dP_n^\Pi}(1-\phi_{i,n})
    \leq e^{-nD_i},
\end{equation}
for large enough $n$, then any $V\subset\bigcup_{1\leq i\leq N} V_i$
receives posterior mass zero asymptotically,
\begin{equation}
  \label{eq:aszero}
  \Pi(V|X_1,\ldots,X_n)\convas{P_0}0.
\end{equation}
\end{lemma}
The condition that covers of the model have to be of \emph{finite}
order is restrictive: problems arise already in parametric context,
for instance, if the $V_i$ are associated with fixed-radius metric balls
required to cover all of $\RR^k$. In such cases application of the
theorem requires a bit more refinement, for example through the
methods put forth in \cite{LeCam73} (see example~\ref{ex:lecamdim}).
Additionally we consider two other alternatives to
by-pass the finiteness assumption on the order of the cover in
section~\ref{sec:sep}.

\subsection{Existence and power of test sequences}
\label{sub:test}

Le~Cam \cite{LeCam73,LeCam75,LeCam86} and Birg\'e
\cite{Birge83,Birge84} put forth a versatile approach to
testing that combines the minimax theorem with the Hellinger geometry
of the model, in particular its Hellinger metric entropy
numbers. Below, we make a carefully chosen variation on this theme
that is technically close to the methods of \cite{Kleijn06}.
(Define $V^n=\{P^n:P\in V\}$ and denote its convex
hull by $\co{V^n}$; elements from $\co{V^n}$ are denoted $P_n$.)
\begin{lemma}
\label{lem:ntests}
Let $n\geq1$, $V\in\scrB$ be given; assume that
$P_0^n({dP^n}/{dP_n^\Pi})<\infty$ for all $P\in V$. Then there exists a test
sequence $(\phi_n)$ such that,
\begin{equation}
   \label{eq:testpower}
  P_0^n\phi_n +
  \sup_{P\in V} P_0^n \frac{dP^n}{dP_n^\Pi}(1-\phi_n)
  \leq \sup_{P_n\in\co{V^n}}\vinf_{\phantom{|}0\leq\al\leq1\phantom{|}}
      P_0^n\Bigl(\frac{dP_n}{dP_n^\Pi}\Bigr)^\al,
\end{equation}
\ie\ testing power is bounded in terms of \emph{Hellinger transforms}.
\end{lemma}
With the next definition, we localize the prior in a flexible sense
and cast the discussion into a frame that also features centrally in
Wong and Shen \cite{Wong95}; where their 
approximation of $P_0$ pertains to a sieve, here it is required
that the set $B$ approximate $P_0$ in the same technical sense.
Given $\Pi$ and a measurable $B$ such that $\Pi(B)>0$,
define the \emph{local} prior predictive distributions $P_n^{\Pi|B}$ by
conditioning the prior predictive on $B$:
\begin{equation}
  \label{eq:locpriorpred}
  P_n^{\Pi|B}(A) = \int Q^n(A)\,d\Pi(Q|B),
\end{equation}
for all $n\geq1$ and $A\in\sigma(X_1,\ldots,X_n)$. Barron localizes
his \emph{matching} criterion in a similar way \cite{Barron86} and
Walker defines \emph{restricted predictive densities} to localize
his analysis \cite{Walker04}. The following lemma
formulates a more easily accessible upper bound for the right-hand
side of inequality~(\ref{eq:testpower}), which prescribes the
($n$-independent) form of the central requirement of theorem~\ref{thm:main}. 
\begin{lemma}
\label{lem:HTBnd}
Let $\Pi$ be given, fix $n\geq1$. Let $V,B\in\scrB$ be such
that $\Pi(B)>0$ and for all $P\in V$, $\sup_{Q\in B} P_0(dP/dQ)<\infty$.
Then there exists a test function $\phi_n:\scrX^n\rightarrow[0,1]$ such that,
\begin{equation}
  \label{eq:iidpower}
  \begin{split}
  P_0^n\phi_n &+ \sup_{P\in V} P_0^n \frac{dP^n}{dP_n^\Pi}(1-\phi_n)\\
    &\leq \vinf_{\phantom{|}0\leq\al\leq1}\Pi(B)^{-\al} 
      \int\biggl[\sup_{P\in \co{V}} P_0\Bigl(\frac{dP}{dQ}\Bigr)^{\al}\biggr]^n
      d\Pi(Q|B).
  \end{split}
\end{equation}
\end{lemma}
\noindent Theorem~\ref{thm:main} is the conclusion of
lemma~\ref{lem:cons} upon substitution of lemmata~\ref{lem:ntests} and
\ref{lem:HTBnd}.


\section{Variations on Schwartz's theorem}
\label{sec:variations}

In this section we apply theorem~\ref{thm:main} to re-derive
Schwartz's theorem, sharpen its assertion to consistency in
Kullback-Leibler divergence and we consider model conditions
that allow priors charging metric balls rather than Kullback-Leibler
neighbourhoods.

\subsection{Schwartz's theorem and Kullback-Leibler priors}
\label{sub:KLpriors}

The strategy to prove posterior consistency in a certain topology
(or more generally, to prove posterior concentration outside a set $V$)
now runs as follows: one looks for a finite cover of $V$ by model subsets
$V_i$, ($1\leq i\leq N$) satisfying the inequalities (\ref{eq:power})
for subsets $B_i$ that are as large as possible and neighbourhoods
of $P_0$ in an appropriate sense. Subsequently
we try to find (a $\sigma$-algebra $\scrB$ on $\scrP$ and) a
prior $\Pi:\scrB\rightarrow[0,1]$ such that ($B_i\in\scrB$ and)
$\Pi(B_i)>0$ for all $1\leq i\leq N$. In this regard the following
lemma offers a great deal of guidance, because it relates testing power
to Kullback-Leibler separation of the sets $B$ and $W$ in
definition~(\ref{eq:testingpower}).
\begin{lemma}
\label{lem:KLequiv}
Let $P_0\in B\subset\scrP$ and $W\subset\scrP$ be given and assume that
there exists an $a\in(0,1)$ such that for all $Q\in B$ and $P\in W$, 
$P_0(dP/dQ)^a<\infty$. Then,
\begin{equation}
  \label{eq:powerbnd}
  \pi_{P_0}(W,B)<1,
\end{equation}
if and only if,
\begin{equation}
  \label{eq:KLsep}
  \sup_{Q\in B}-P_0\log\frac{dQ}{dP_0}
    < \vinf_{\phantom{|}P\in W}-P_0\log\frac{dP}{dP_0}.
\end{equation}
\end{lemma}
(Note that in applications of lemma~\ref{lem:KLequiv} the sets $W_i$ are
\emph{convex hulls} of model subsets $V_i$.)
Due to the fact that Kullback-Leibler divergence dominates Hellinger
distance, the proof of Schwartz's theorem is now immediate (at least,
for models that have $P_0(dP/dQ)<\infty$ for all $P\in V$
and all $Q$ in a Kullback-Leibler neighbourhood of $P_0$ that is small
enough.)
It is clear that Schwartz's theorem does not fully exploit the room that
(\ref{eq:KLsep}) offers because it does not prove posterior consistency
in Kullback-Leibler divergence. The following theorem provides such
an assertion without requiring more of the prior.
\begin{theorem}
\label{thm:KLDcons}
Let $P_0$ and the model be such that for some Kullback-Leibler
neighbourhood $B$ of $P_0$, 
$\sup_{Q\in B}P_0(dP/dQ)<\infty$ for all $P\in\scrP$. Let $\Pi$
be a Kullback-Leibler prior. For any $\ep>0$,
assume that $\{P\in\scrP:-P_0\log(dP/dP_0)\geq\ep\}$
is covered by a finite number $N\geq1$ of model subsets $V_1,\ldots,V_N$
such that,
\begin{equation}
  \label{eq:coKLD}
  \inf_{P\in\co{V_i}} -P_0\log\frac{dP}{dP_0} > 0,
\end{equation}
for all $1\leq i\leq N$. Then for \iid-$P_0$ distributed $X_1,X_2,\ldots$,
\begin{equation}
  \label{eq:KLDcons}
  \Pi\bigl(\,P\in\scrP\,:\,-P_0\log(dP/dP_0)<\ep
    \bigm| X_1,\ldots,X_n\,\bigr)\convas{P_0}1.
\end{equation}
\end{theorem}
Compare this formulation with theorem~2 of \cite{Walker04}, which also
asserts consistency in Kullback-Leibler divergence. Although Walker's proof
depends on completely different techniques, the Kullback-Leibler
separation condition of lemma~\ref{lem:KLequiv} appears there as well
(but in a different form, as a limiting lower bound for the Kullback-Leibler
divergence of Walker's `restricted predictive distributions' with
respect to $P_0$). The occurrence of Kullback-Leibler separation from
two perspectives that are technically so remote can be viewed as
mutually supportive and underlines the conclusion that Kullback-Leibler
separation is a fundamental aspect in this context.

To appreciate how a finite cover of Kullback-Leibler-neighbourhoods
may occur in models, consider the following example that relies on relative
compactness with respect to the uniform norm for log-densities.
\begin{example}
Let $\ep>0$ be given and
assume that the complement $V$ of a Kullback-Leibler ball of radius
$\ep>0$ contains $N$ points $P_1,\ldots,P_N$ such that the convex sets,
\[
  V_i=\bigl\{P\in\scrP:\|dP/dP_i-1\|_{\infty}<\ft12\ep\bigr\},
\]
cover $V$. Finiteness of the cover can be guaranteed, for example,
if the model describes data taking values in a fixed
bounded interval in $\RR$ and the associated family of log-densities
is bounded and equicontinuous, by virtue of the Ascoli-Arzel\`a
compactness theorem. (Other ways to find suitable covers refer to
$\|\cdot\|_{\infty}$-entropy or bracketing numbers for log-likelihood
ratios \cite{vdVaart96}.) Then any $P\in \co{V_i}$ satisfies
$\|dP/dP_i-1\|_{\infty}<\ft12\ep$ as well, and hence,
$\log(dP/dP_i)\leq\log(1+\ft12\ep)\leq\ft12\ep$. As a result,
\[
  -P_0\log\frac{dP}{dP_0} \geq \ep - P_0\log\frac{dP}{dP_i} \geq \ft12\ep,
\]
and (\ref{eq:coKLD}) holds. In such models, any prior $\Pi$ satisfying
(\ref{eq:KLprior}) leads to a posterior that is consistent in
Kullback-Leibler divergence.
\end{example}

\subsection{Priors that charge metric balls}
\label{sub:hellpriors}

Quite generally, lemma~\ref{lem:KLequiv} shows that model
subsets are consistently testable {\it if and only if} they can be
separated from neighbourhoods of $P_0$ in Kullback-Leibler
divergence. This illustrates the fundamental nature of Schwartz's prior
mass requirement and undermines hopes for useful priors that charge
different neighbourhoods of $P_0$  in general. However,
lemma~\ref{lem:KLequiv} does not cover all situations and does not
exclude the possibility of gaining freedom in the choice of the prior
by strengthening requirements on the model. In this subsection, we
give examples of `metric' priors together with model conditions that 
validate them in the sense of asymptotic posterior consistency.

Initially, given ($P_0$ and) a suitable neighbourhood $B$, we impose
that for all $Q\in B$ and any $P\in \scrP$, $p/q\in L_2(Q)$ (with
norm denoted $\|\cdot\|_{2,Q}$).
Under this condition the Cauchy-Schwarz inequality leads to,
\[
  \begin{split}
  P_0\Bigl(\frac{p}{q}\Bigr)^{1/2}
    &= \int \Bigl(\frac{p_0}{q}\Bigr)^{1/2}\,p_0^{1/2}p^{1/2}\,d\mu\\
    &= \int p_0^{1/2}\,p^{1/2}\,d\mu
      -\int\Bigl(1-\Bigl(\frac{p_0}{q}\Bigr)^{1/2}\Bigr)\,
      \Bigl(\frac{p_0}{q}\Bigr)^{1/2}\Bigl(\frac{p}{q}\Bigr)^{1/2}\,dQ\\
    &\leq  1-\frac{1}{2}H(P_0,P)^2
      +H(P_0,Q)\, \Bigl\|\frac{p_0}{q}\Bigr\|_{2,Q}^{1/2}\,
        \Bigl\|\frac{p}{q}\Bigr\|_{2,Q}^{1/2}.
  \end{split}
\]
To enable the use of priors that charge Hellinger balls instead of
KL-neighbourhoods, we strengthen
the above bound to a uniform bound over the model, making it possible
to separate $B$ from $V$ in Hellinger distance to prove existence of
uniform tests. Combined with lemma~\ref{lem:cons} this leads to the
following theorem.
\begin{theorem}
\label{thm:hell}
Let the model $\scrP$ be totally
bounded with respect to the Hellinger metric. Assume also that there
exists a constant $L>0$ and a Hellinger ball $B'$ centred on $P_0$ such
that for all $P\in\scrP$ and $Q\in B'$,
\begin{equation}
  \label{eq:elltwoball}
  \Bigl\|\frac{p}{q}\Bigr\|_{2,Q} = \Bigl(\int \frac{p^2}{q}\,d\mu\Bigr)^{1/2}
  < L.
\end{equation}
Finally assume that for any Hellinger neighbourhood $B$ of $P_0$,
$\Pi(B)>0$. Then the posterior is Hellinger consistent, $P_0$-almost-surely.
\end{theorem}
As a side-remark, note that it is possible that (\ref{eq:elltwoball})
is not satisfied without extra conditions on $Q$. In that case impose
that $B$ is \emph{included} in a Hellinger ball, while satisfying other
conditions as well; the theorem remains valid as long as we also change
the prior, \ie\ as long as $\Pi(B)>0$ is maintained.

Varying on this theme, choose $1\leq r<\infty$. Analogous to the
Hellinger metric ($r=2$), define, for
all $P,Q$ probability measures, Matusita's \emph{$r$-metric distance}
\cite{Matusita71},
\[
  d_r(P,Q) = \Bigl(\int \bigl|\,p^{1/r}-q^{1/r} \bigr|^rd\mu \Bigr)^{1/r},
\]
(based on any $\sigma$-finite $\mu$ that dominates $P$ and $Q$).
Applying H\"older's inequality where we applied Cauchy-Schwarz before
and dominating the constant of the rest-term in a different way, we arrive
at the following theorem concerning priors that charge $d_r$-balls.
\begin{theorem}
\label{thm:drcons}
Let $1\leq r<\infty$ be given and let the model $\scrP$ be
a totally bounded metric space with respect to $d_r$. Let $X_1, X_2,\ldots$
be \iid-$P_0$ distributed for some $P_0\in\scrP$. Assume that the prior
is such that $P_0^n \ll P_{n}^{\Pi}$, for all $n\geq1$ and satisfies,
\begin{equation}
  \label{eq:drprior}
  \Pi\bigl(\,P\in\scrP\,:\,d_r(P_0,P)<\delta\,\bigr) > 0,
\end{equation}
for all $\delta>0$. In addition, assume that there is an $L>0$ and
a $d_r$-ball $B$ such that for all $P\in\scrP$ and $Q\in B$,
$P_0({p/q})^{{s/r}\vee1} \leq L^s$, where $1/r+1/s=1$. Then the posterior
is consistent in the $d_r$-metric, $P_0$-almost-surely.
\end{theorem}
\begin{remark}
\label{rem:netprior}
For the models under discussion, we note the following general
construction of so-called \emph{net priors}
\cite{LeCam72b,Ghosal97,Ghosal00,Kleijn04}: denote the metric on
$\scrP$ by $d$. Initially, assume that $\scrP$ has finite $d$-metric
entropy numbers. Let $(\eta_m)$ be any sequence such that
$\eta_m>0$ for all $m\geq1$ and $\eta_m\downarrow0$. For fixed $m\geq1$,
let $P_1,\ldots,P_{M_m}$ denote an $\eta_m$-net for $\scrP$
and define $\Pi_m$ to be the measure that places mass $1/M_m$ at every
$P_i$, ($1\leq i\leq M_m$). Choose a sequence $(\lambda_m)$ such that
$\lambda_m>0$ for all $m\geq1$ and $\sum_{m\geq1}\lambda_m=1$, to define
the net prior $\Pi=\sum_{m\geq1}\lambda_m\,\Pi_m$.
In case $\scrP$ is not totally bounded, one may generalize the above
construction by choosing an increasing sequence $(K_m)$ of compact
submodels, each of which is totally bounded so that for every
$m\geq1$, a $\Pi_m$ with finite support inside $K_m$ can be defined
like above. Any net prior is Radon by construction and if $\scrP$ is
totally bounded (or, if $\scrP$ is separable and $\scrP$ equals the
closure of $\cup_mK_m$) a net prior assigns non-zero mass to every
open set. In addition, lower-bounds for prior mass in metric balls
are proportional to inverses of upper bounds for metric entropy
numbers, provided we choose $(\lambda_m)$ appropriately, which is
very helpful when one is interested in rates of convergence
\cite{Ghosal00,Kleijn04}. In the Hellinger case, a net prior satisfies
(\ref{eq:dompriorpred}) and theorem~\ref{thm:hell} applies if
(\ref{eq:elltwoball}) holds.
\end{remark}
Net priors, or more generally, Borel priors of \emph{full support}
(that is, $\Pi(U)>0$ for every open $U\subset\scrP$) are also helpful
if one is interested in the construction of Kullback-Leibler priors,
at least, if the corresponding topology is fine enough.
\begin{lemma}
\label{lem:KLcont}
Let $\scrP$ be a topological space. If for every $P\in\scrP$,
the Kullback-Leibler divergence $\scrP\rightarrow\RR:Q\mapsto-P\log(dQ/dP)$
is continuous, then a Borel prior of full support is a Kullback-Leibler
prior. If, in addition, $\scrP$ is metrizable, all net priors of
full support are Kullback-Leibler priors.
\end{lemma}
When discussing consistency, requirements on the model like
(\ref{eq:elltwoball}) are present to guarantee continuity of the
Kullback-Leibler-divergence. For example, the perceptive reader
may have recognized in
(\ref{eq:elltwoball}) sufficiency to invoke theorem~5 of \cite{Wong95}
which provides an upper bound for the Kullback-Leibler divergence
in terms of the Hellinger distance. The latter is a stronger,
Lipschitz-like variation on the continuity condition of the above
lemma.


\section{Posterior consistency on separable models}
\label{sec:sep}

Requiring \emph{finiteness} of the order of the cover in
theorem~\ref{thm:main} and lemma~\ref{lem:cons} is somewhat crude.
Besides Le~Cam's construction of example~\ref{ex:lecamdim}, there are
several ways out: firstly, in subsection~\ref{sub:barron} we explore
the possibility of letting a sieve of totally bounded submodels
approximate the full model analogous to Barron's theorem. Secondly,
Hellinger consistency of the posterior on separable models formed
the assertion of a remarkable theorem of Walker for a Kullback-Leibler
prior that also satisfies a summability condition \cite{Walker04}. In
subsection~\ref{sub:walker} we show that variations on Walker's
theorem can be derived with the methods of section~\ref{sec:cons}.

\subsection{Generalization to sieves}
\label{sub:barron}

If the prior is Radon (\eg\ when the model is a Polish space), inner
regularity says that the model can be approximated in prior measure
by compact submodels. Since the latter are totally bounded, a proof is
conceivable based on an approximating sieve of relatively compact submodels.
If we require that the ingredients of the above argument satisfy
certain bounds, a theorem of this nature is possible.
\begin{theorem}
\label{thm:barron}
Let $X_1, X_2, \ldots$ be $\iid-P_0$ for some $P_0\in\scrP$ and let
$V$ be given. Assume that $P_0^n\ll P_n^\Pi$ for all $n\geq1$ and that
there exist constants $K,L>0$ and a
sequence of submodels $(\scrP_n)$ such that for large
enough $n\geq1$,
\begin{itemize}
\item[(i.)]
there is a cover $V_1,\ldots,V_{N_n}$ for $V\cap\scrP_n$ of order
$N_n\leq\exp(\ft12Ln)$ with tests $\phi_{1,n},\ldots,\phi_{N_n,n}$
such that,
\[
P_0^n\phi_{i,n} +
  \sup_{P\in V_i} P_0^n \frac{dP^n}{dP_n^\Pi}(1-\phi_{i,n})
    \leq e^{-nL},
\]
for all $1\leq i\leq N_n$;
\item[(ii.)] the prior mass $\Pi(\scrP\setminus\scrP_n)\leq\exp(-nK)$ and,
\begin{equation}
  \label{eq:barron}
  \sup_{P\in V\setminus\scrP_n}\sup_{Q\in B}
    P_0\Bigl(\frac{dP}{dQ}\Bigr)\leq e^{\ft{K}2},
\end{equation}
for some model subset $B$ such that $\Pi(B)>0$.
\end{itemize}
Then $\Pi(\,V\,|\,X_1,\ldots,X_n\,)\convas{P_0}0$.
\end{theorem}
Condition {\it (i.)} of theorem~\ref{thm:barron} represents
condition (\ref{eq:testcond}) in the present context, embedding
the construction illustrated previously in a sequence of submodels
$\scrP_n$. Consequently existence proofs for tests and upper bounds
for testing power of the preceding subsections remain applicable. More
particularly, condition {\it (i.)} has the following alternative.
\begin{itemize}
\item[{\it (i'.)}] 
{\it there exist a model subset $B$ with $\Pi(B)>0$ and a cover
$V_1,\ldots,V_{N_n}$ for $V\cap\scrP_n$ of order $N_n\leq\exp(\ft12Ln)$,
such that for every $1\leq i\leq N_n$,
\[
  \pi_{P_0}\bigl(\,\co{V_i},B\,)\leq e^{-L},
\]
and $\sup_{Q\in B} P_0(dP/dQ)<\infty$ for all $P\in V_i$.}
\end{itemize}
Condition {\it (ii.)} of theorem~\ref{thm:barron} requires negligibly
small prior mass outside the sieve, where `negligibility' is
determined by inequality (\ref{eq:barron}). If we think of $B$ as a
small neighbourhood around $P_0$, it appears that the
freedom to choose $B$ enables upper bounds for the \lhs\ of
(\ref{eq:barron}) arbitrarily close to one (\ie\ to satisfy
(\ref{eq:barron}), $K$ can be chosen arbitrarily close to zero).  In
such cases, condition {\it (ii.)} reduces to the requirement that
$\Pi(\scrP\setminus\scrP_n)$ decreases exponentially, which is
Barron's original
requirement on the prior mass outside the sieve (see, for example,
\cite{Barron88a}). The following example illustrates this point.
\begin{example}
Assume that $X_1,X_2,\ldots$ are \iid-$P_0$ for some $P_0$ in a
model $\scrP$ that is dominated by a $\sigma$-finite measure $\mu$. 
Consider a prior $\Pi$ that charges all $L_\infty(\mu)$-balls around
$\log p_0$ (where $p_0,p$ denote the $\mu$-densities for $P_0,P$
respectively):
\[
  \Pi\bigl(\,P\in\scrP\,:\,\|\log p - \log p_0\|_{\infty}<\ep\,\bigr)>0,
\]
for all $\ep>0$. Note that, for all $P\in\scrP$,
\[
  P_0\Bigl(\frac{dP}{dQ}\Bigr)=\int\frac{p_0\,p}{q}\,d\mu
    =\int\frac{p_0}{q}\,dP\leq e^\ep,
\]
whenever $\|\log q - \log p_0\bigr\|_{\infty}\leq\ep$. Hence, a sieve
$(\scrP_n)$ satisfying condition {\it (i.)} such that
$\Pi(\scrP\setminus\scrP_n)\leq\exp(-nK')$ for \emph{some} small
$K'>0$ would suffice in this case and similar ones.
\end{example}
A generalization of condition {\it (ii.)} of theorem~\ref{thm:barron}
involving $n$-dependent choices for $B$ can be found in
appendix~\ref{app:proofs}. Theorem~\ref{thm:barron} is applied in
the support boundary problem of section~\ref{sec:boundary}, see
remark~\ref{rem:barron}.

\subsection{Variations on Walker's theorem}
\label{sub:walker}

In this subsection we abandon constructions based on finite covers 
altogether and require only that the cover is \emph{countable}.  
A natural setting arises when we consider models that are
\emph{separable} in some metric topology, in which case countable
covers by balls of any radius exist.
Like theorem~\ref{thm:barron}, the most notable change
in perspective that the relevant consistency theorem implies, is that,
aside from \emph{lower} bounds for prior mass (\eg\
Kullback-Leibler-priors, net priors, etc.), conditions also include an
\emph{upper} bound. 
\begin{theorem}
\label{thm:sep}
Let $\scrP$ and $\Pi$ be given and
assume that $P_0^n\ll P_n^\Pi$ for all $n\geq1$. Let $V$ be a
model subset, with a countable cover $V_1,V_2,\ldots$
and $B_1,B_2,\ldots$ such that for all $i\geq1$,
we have $\Pi(B_i)>0$ and for all $P\in V_i$,
$\sup_{Q\in B_i} P_0(dP/dQ)<\infty$. Then,
\begin{equation}
  \label{eq:prewalker}
  P_0^n \Pi(V|X_1,\ldots,X_n) 
  \leq \sum_{i\geq1} \inf_{0\leq\al\leq1} \frac{\Pi(V_i)^\al}{\Pi(B_i)^\al}
    \pi_{P_0}\bigl(\co{V_i}, B_i;\al\bigr)^n.
\end{equation}
\end{theorem}
The following two corollaries show how theorem~\ref{thm:sep} is
related to Walker's theorem~\ref{thm:walker}. The first is based on 
a prior that satisfies a condition of the form~(\ref{eq:walker}),
but does \emph{not} make the assumption that the prior is also
a Kullback-Leibler prior yet. Instead, a model condition with a
role similar to that of (\ref{eq:power}) is imposed. 
\begin{corollary}
\label{cor:varwalker}
Let $\scrP$ and $\Pi$ be given and
assume that $P_0^n\ll P_n^\Pi$ for all $n\geq1$. Let $V$ be a
model subset, with a countable cover $V_1,V_2,\ldots$.
and a $B\subset\scrP$ such that $\Pi(B)>0$ and for all
$i\geq1$, $P\in V_i$, $\sup_{Q\in B} P_0(dP/dQ)<\infty$. Furthermore,
assume that,
\begin{equation}
  \label{eq:tstpwr}
  \sup_{i\geq1} \sup_{P\in\co{V_i}} \sup_{Q\in B}
    P_0\Bigl(\frac{dP}{dQ}\Bigr)^{1/2} < 1.
\end{equation}
If the prior satisfies the summability condition,
\begin{equation}
  \label{eq:walkercond}
  \sum_{\phantom{|}i\geq1} \Pi(V_i)^{1/2} < \infty,
\end{equation}
then the posterior satisfies, $\Pi(V|X_1,\ldots,X_n)\convas{P_0}0$.
\end{corollary}
The second corollary does not impose model conditions like
(\ref{eq:tstpwr}), and, instead, requires a Kullback-Leibler prior
that satisfies a summability condition that is slightly stronger than
condition (\ref{eq:walker}) in theorem~\ref{thm:walker}.
\begin{corollary}
\label{cor:onlysum}
Let $\scrP$ be separable in the Hellinger topology. Assume that
there is Kullback-Leibler neighbourhood $B$ of $P_0$ such that for
all $P\in\scrP$, $\sup_{Q\in B} P_0(dP/dQ)<\infty$.
Let $\Pi$ be a Kullback-Leibler prior such that for all $\beta>0$,
\begin{equation}
  \label{eq:walkerprime}
  \sum_{i\geq1} \Pi(V_i)^{\beta} < \infty,
\end{equation}
where the $V_i$, ($i\geq1$) are any cover of $\scrP$ by Hellinger
balls of a fixed radius. Then the posterior is $P_0$-almost-surely
Hellinger consistent.
\end{corollary}


\section{Posterior rates of convergence}
\label{sec:rates}

Minimax rates of convergence for (estimators based on) posterior
distributions were considered more or less simultaneously in
\cite{Ghosal00} and \cite{Shen01}, with conditions that display very
close resemblance. Both pose (\ref{eq:GGV}) as the condition on the
prior and both appear to be inspired by contemporary results
regarding Hellinger rates of convergence for sieve MLE's, as well as
on \cite{Barron99}, which concerns posterior consistency based on
controlled bracketing entropy for a sieve, up to subsets of negligible
prior mass, following ideas that were first laid down in
\cite{Barron88a}. (Although formulated for fixed-radius Hellinger balls,
it is remarked already in \cite{Barron99} that their main theorem can
also be formulated for $\ep_n\downarrow0$, with reference to
\cite{Shen01}.) More recently, Walker, Lijoi and Pr\"unster
\cite{Walker07} have added to these considerations with a theorem
for Hellinger rates of posterior concentration in models that are
separable for the Hellinger metric, with a central condition that
calls for summability of square-roots of prior masses of covers
of the model by Hellinger balls, based on analogous consistency
results in \cite{Walker04}.

Note that methods proposed in the preceding sections hold at
finite values of $n\geq1$: the hypothesis $B,V$ as well as the
constant $\al$ can be made $n$-dependent without changing the
basic building blocks. As such, not much needs to be adapted to
preceding results to extend also to rates of posterior convergence.
Below we follow Barron's ideas again and sharpen
theorem~\ref{thm:barron} to accomodate rates of posterior convergence.
For the theorem below, we endow the model
with a metric $d$ and assume that the prior is Borel with respect
to the associated metric topology.
\begin{theorem}
\label{thm:sieve}
Let $X_1, X_2, \ldots$ be $\iid-P_0$ for some $P_0\in\scrP$. Assume
that the prior $\Pi$ is such that $P_0^n\ll P^{\Pi}_n$ for all
$n\geq1$. Let $(\ep_n)$ be a sequence with $\ep_n\downarrow 0$ and
$n\ep_n^2\rightarrow\infty$. Define
$V_n=\{P\in\scrP:d(P,P_0)>\ep_n\}$, a sequence of measurable
submodels $\scrP_n\subset\scrP$ and measurable model subsets $B_n$
such that $\sup_{Q\in B_n}P_0(dP/dQ)<\infty$ for all $P\in V_n$.
Assume that, for sufficiently large $n\geq 1$,
\begin{itemize}
\item[(i)]
  there is an $L>0$ such that $V_n\cap\scrP_n$ has a cover
  $V_{n,1},V_{n,2},\ldots,V_{n,N_n}\subset\scrP_n$
  of order $N_n\leq \exp(\ft12 L n\ep_n^2)$, such that for
  all $1\leq i\leq N_n$,
  \begin{equation}
    \label{eq:pitwo}
    \pi_{P_0}\bigl(\,\co{V_{n,i}},B_n\,\bigr)\leq e^{-L\ep_n^2},
  \end{equation}
\item[(ii)]
  there is a $K>0$ such that 
  $\Pi(\scrP\backslash\scrP_n) \leq e^{-K n \ep_n^2}$
  and $\Pi(B_n)\geq e^{-\frac{K}{2} n \ep_n^2}$, while also,
  \begin{equation}
  \label{eq:sievemodelcond}
    \sup_{P\in\scrP\setminus\scrP_n} \sup_{Q\in B_n}
    P_0\Bigl(\frac{dP}{dQ}\Bigr) < e^{\frac{K}{4}\ep_n^2}.
  \end{equation}
\end{itemize}
Then,
$\Pi(\,P\in\scrP:\,d(P,P_0)>\ep_n\,|\,X_1,\ldots,X_n\,)
\convprob{P_0}0$.
\end{theorem}
This theorem has been formulated generally and this
generality obscures the interpretation of conditions somewhat: the
first condition plays the same role as the entropy condition in the
Ghosal-Ghosh-van~der~Vaart theorem; it enables construction of a
suitable minimax test.
Sufficiency of prior mass around $P_0$ forms part of the second
condition, which also assures that the sieve approximates the model
closely enough, by upper-bounding prior mass outside the sieve.
However, both conditions do not illustrate these points with clarity,
so our first goal is to indicate how theorem~\ref{thm:sieve} relates
to more familiar conditions.

Under a mild integrability condition, condition
(\ref{eq:pitwo}) for the sets $\co{V_{n,i}}$ and $B_n$ follows from
a minimal amount of separation of $\co{V_{n,i}}$ and $B_n$ in
Kullback-Leibler divergence. 
\begin{lemma}
\label{lem:KLseparation}
Consider two model subsets $B,W$ such that $P_0\in B$. Suppose that
for some $a\in(0,1)$, $P_0(dP/dQ)^a$ is finite for all $P\in W$, $Q\in
B$. If, for some $\Delta>0$,
\begin{equation}
  \label{eq:KLsepDelta}
  \sup_{Q\in B}-P_0\log\frac{dQ}{dP_0}
  \leq \inf_{P\in W}-P_0\log\frac{dP}{dP_0}-\Delta,
\end{equation}
then there exists an $\alpha\in(0,1)$ such that,
\[
  \pi_{P_0}(B,W) \leq e^{-\alpha\Delta}.
\]
Conversely, if for some $\Delta>0$,
\[
  \sup_{Q\in B}-P_0\log\frac{dQ}{dP_0}
  > \inf_{P\in W}-P_0\log\frac{dP}{dP_0}-\Delta,
\]
then $\pi_{P_0}(B,W;\al) > e^{-\alpha\Delta}$ for all $\al\in(0,1)$.
\end{lemma}
Lemma~\ref{lem:KLseparation} says that if $B$ and $W$ are separated
in Kullback-Leibler divergence by some small difference $\Delta$,
then the logarithm of the Hellinger transform $\log\pi_{P_0}(B,W)$ is
upper-bounded by a multiple of $-\Delta$. This emphasizes the
fundamental role played by the
Kullback-Leibler divergence and it illustrates the associated
limitations: not all models have integrable likelihood ratios,
and Kullback-Leibler divergences that are infinite make inequality
(\ref{eq:KLsepDelta}) void.

With lemma~\ref{lem:KLseparation} in hand, we can simplify and specify
theorem~\ref{thm:sieve} considerably, to bring us closer to
the Ghosal-Ghosh-van~der~Vaart theorem.
\begin{corollary}
\label{cor:altGGV}
Let $X_1, X_2, \ldots$ be \iid-$P_0$ for some $P_0\in\scrP$. 
Specify that the metric on $\scrP$ is the Hellinger metric $H$;
define $(\ep_n)$ with $\ep_n\downarrow0$ and
$n\ep_n^2\rightarrow\infty$, and take
$V_n=\{P\in\scrP:H(P_0,P)>M\ep_n\}$, for $M>0$, and
$B_n=\{Q\in\scrP:-P_0\log(dQ/dP_0)<\ep_n^2\}$. Assume that
for $n$ large enough and all $P\in V_n$, $\sup\{P_0(dP/dQ):Q\in
B_n\}<\infty$. If, for large enough $n\geq1$,
\begin{itemize}
\item[(i)] there is an $L>0$, such that
  $N(\ep_n,\scrP,H)\leq e^{Ln\ep_n^2}$;
\item[(ii)] there is a $K>0$, such that 
  \begin{equation}
  \label{eq:altGGV}
    \Pi\Bigl(\,P\in\scrP\,:\,-P_0\log\frac{dP}{dP_0}<\ep_n^2\,\Bigr)
    \geq e^{-K n\ep_n^2},
  \end{equation}
\end{itemize}
then $\Pi(\,P\in\scrP:\,H(P,P_0)>M\ep_n\,|\,X_1,\ldots,X_n\,)
\convprob{P_0}0$, for $M$ large enough.
\end{corollary}
Comparison with theorem~\ref{thm:GGV} shows that the requirement on
the prior is now formulated
in terms of Schwartz's KL-neighbourhoods rather than the second-order
neighbourhoods of (\ref{eq:GGV}), at the expense of an
integrability condition. For an analysis of example~\ref{ex:noGGVpriors}
using corollary~\ref{cor:altGGV}, see example~\ref{ex:heavytail}.


\section{Marginal consistency}
\label{sec:boundary}

Semi-parametric statistics presents a
well-developed frequentist theory of finite-dimensional parameter
estimation in infinite-dimensional models, including notions of
optimality for parameters that are smooth functionals of model
distributions. By comparison, Bayesian semi-parametric methods are
still in the early stages of development \cite{Bickel12}.
In this section a method of demonstrating marginal consistency is
formulated, based on the material in preceding sections.

The basic problem is set as follows: let $\Tht$
be an open subset of $\RR^k$ parametrizing the \emph{parameter of
interest} $\tht$ and let $H$ be a measurable (and typically
infinite-dimensional) parameter space for the
\emph{nuisance parameter} $\eta$.
The model is $\scrP=\{P_{\tht,\eta}:\tht\in\Tht,\eta\in H\}$ where
$\Tht\times H\to\scrP:(\tht,\eta)\mapsto P_{\tht,\eta}$ is a Markov kernel
on the sample space $(\scrX,\scrA)$ describing the distributions
of individual points from an infinite \iid\ sample $X_1,X_2,
\ldots\in\scrX$. Given a metric $g:\Tht\times\Tht\rightarrow[0,\infty)$
and a prior measure $\Pi$ on $\Tht\times H$ we say that the posterior
is \emph{marginally consistent} for the parameter of interest,
if for all $\ep>0$,
\begin{equation}
  \label{eq:marg}
  \Pi\bigl(\,P_{\tht,\eta}\in\scrP:g(\tht,\tht_0)>\ep,\eta\in H
    \,\bigm|\,X_1,\ldots,X_n\,\bigr) \convas{P_{\tht_0,\eta_0}}0,
\end{equation}
for all $\tht_0\in\Tht$ and $\eta_0\in H$. Marginal consistency
amounts to consistency with
respect to the pseudo-metric $d:\scrP\times\scrP\rightarrow[0,\infty)$,
$d\bigl(P_{\tht,\eta},P_{\tht',\eta'}\bigr) = g(\tht,\tht')$,
for all $\tht,\tht'\in\Tht$ and $\eta,\eta'\in H$. The following
theorem is a formulation of theorem~\ref{thm:main} specific to marginal
consistency.
\begin{theorem}
\label{thm:marginal}
Let $\scrP=\{P_{\tht,\eta}:\tht\in\Tht,\eta\in H\}$ be a model
for data $X_1, X_2, \ldots$ assumed distributed \iid-$P_0$ for
some $P_0\in\scrP$ in the Hellinger support of $\Pi$. Let
$\ep>0$ be given,
define $V=\{P_{\tht,\eta}\in\scrP:g(\tht,\tht_0)>\ep,\eta\in H\}$ and assume
that $V_1,\ldots,V_N$ form a finite cover of $V$. If there exist
model subsets $B_1,\ldots, B_N$ such that for every $1\leq i\leq N$,
\[
  \pi_{P_0}(\,\co{V_i},\,B_i\,) < 1,
\]
$\Pi(B_i)>0$ and $\sup_{Q\in B_i} P_0(dP/dQ)<\infty$ for all $P\in V_i$,
then the posterior is marginally consistent, $P_0$-almost-surely.
\end{theorem}

\subsection{Density support boundaries}

Consistent support boundary estimation (see \cite{Ibragimov81},
or \cite{Reiss17} for a more recent, Bayesian reference),
though easy from the perspective of
point-estimation, is not a triviality when using Bayesian
methods because one is required to specify a nuisance space
\cite{Ritov14}. The Bernstein-Von~Mises phenomenon for this type
of problem is studied in Kleijn and Knapik \cite{Kleijn15} and
leads to exponential rather than normal limiting form for
the posterior. Below, we prove consistency using
theorem~\ref{thm:main}. 

Consider the following simple model: for some constant $\sigma>0$
define the parameter of interest to lie in the space
$\Tht=\{\tht=(\tht_1,\tht_2)\in\RR^2:0<\tht_2-\tht_1<\sigma\}$ equipped with
the Euclidean norm $\|\cdot\|$. Let $H$ be a collection of Lebesgue
probability densities {$\eta:[0,1]\rightarrow[0,\infty)$} for which
there are a constant $a>0$ and a continuous, monotone increasing
$f:(0,a)\rightarrow(0,\infty)$ with $f(0+)=0$, such that,
\begin{equation}
  \label{eq:lower}
  \inf_{\eta\in H}
  \min\Bigl\{\int_0^\ep\eta\,d\mu, \int_{1-\ep}^1\eta\,d\mu\Bigr\}
  \geq f(\ep),\quad(0<\ep<a).
\end{equation}
The model $\scrP = \{ P_{\tht,\eta}:\tht\in\Tht,\eta\in H\}$ is defined
in terms of Lebesgue densities of the following semi-parametric
form,
\[
  p_{\tht,\eta}(x) =\frac{1}{\tht_2-\tht_1}\,
    \eta\Bigl(\frac{x-\tht_1}{\tht_2-\tht_1}\Bigr)\,
     1_{\{\tht_1\leq x\leq\tht_2\}},
\]
for some $(\tht_1,\tht_2)\in\Tht$ and $\eta\in H$. A condition
like (\ref{eq:lower}) is necessarily part of the analysis,
because questions concerning support boundary points make sense
\emph{only} if the distributions under consideration put mass in
every neighbourhood of $\tht_1$ and $\tht_2$.(Let
$\|\cdot\|_{s,Q}$ denote the $L_s(Q)$-norm, for $s\geq1$.)
\begin{theorem}
\label{thm:domain}
For some $\sigma>0$, let ${\Tht}$ be $\{(\tht_1,\tht_2)\in\RR^2:
0<\tht_2-\tht_1<\sigma\}$ and let the space $H$ with associated
function $f$ as in (\ref{eq:lower}) be given. Assume that there
exists an $s\geq1$ such that the sets $B$,
\[
  B=\Bigl\{Q\in\scrP:\Bigl\|\frac{dP_0}{dQ}-1\Bigr\|_{s,Q}<\delta
    \Bigr\},
\]
satisfy $\Pi(B)>0$ for all $\delta>0$. Also assume there exists a
constant $K>0$ such that for all $P\in\scrP$ and $Q\in B$,
$\|dP/dQ\|_{r,Q}\leq K$, where $1/r+1/s=1$. If $X_1,X_2,\ldots$
form an \iid-$P_0$ sample for $P_0=P_{\tht_0,\eta_0}\in\scrP$
then,
\begin{equation}
  \label{eq:consdomain}
  \Pi\bigl(\,\tht\in\Tht\,:\,\|\tht-\tht_0\|<\ep\bigm|X_1,\ldots,X_n\,\bigr)
  \convas{P_0}1,
\end{equation}
for all $\ep>0$.
\end{theorem}
\begin{example}
\label{ex:domain}
To apply theorem~\ref{thm:domain}, let $P_0=P_{\tht_0,\eta_0}$ 
be a distribution on $\RR$ with Lebesgue density
$p_0:\RR\mapsto[0,\infty)$ supported on an interval
$[\tht_{0,1},\tht_{0,2}]$ of a width smaller than
or equal to a (known) constant $\sigma>0$. Furthermore, let
$g:[0,1]\rightarrow[0,\infty)$ be a known Lebesgue probability density
such that $g(x)>0$ for all $x\in(0,1)$. For some constant
$M>0$ consider the subset $C_M$
of $C[0,1]$ of all continuous $h:[0,1]\rightarrow[0,\infty)$
such that $e^{-M}\leq h\leq e^M$. To define the model's
dependence on the nuisance parameter $h$, let $H$ contain all
$\eta:[0,1]\rightarrow[0,\infty)$ that are Esscher transforms
\cite{Leonard78} of the form,
\[
  \eta(x) = \frac{g(x)\,h(x)}{\int_0^1g(y)\,h(y)\,dy},
\]
for some $h\in C_M$ and all $x\in[0,1]$. To define a prior on $H$,
let $U\sim U[-M,M]$ be uniformly distributed
on $[-M,M]$ and let $W=\{W(x):x\in[0,1]\}$ be Brownian motion on
$[0,1]$, independent of $U$. Note that it is possible to
condition the process $Z(x)=U+W(x)$ on $-M\leq Z(x)\leq M$
for all $x\in[0,1]$ (or reflect $Z$ in $z=-M$ and $z=M$).
Define the distribution of $\eta$ under the prior $\Pi_H$ by
taking $h=e^Z$. On $\Tht$ let $\Pi_{\Tht}$ denote a prior with
a Lebesgue density that is continuous and strictly positive
on $\Tht$. One verifies easily that the model satisfies
(\ref{eq:lower}) with $f$ defined by,
\[
  f(\ep) = e^{-2M}\,\min\Bigl\{ \int_0^\ep g(x)\,dx,
    \int_{1-\ep}^1g(x)\,dx\Bigr\},
\]
for all $\ep>0$ small enough. The prior mass requirement is
satisfied because the distribution of the process
$Z$ has full support relative to the uniform norm in the
collection of all continuous functions on $[0,1]$ bounded by $M$.
\end{example}
\begin{remark}
\label{rem:barron}
If the assumed bound $\sigma>0$ is set to infinity, testing power is
lost (see the proof of theorem~\ref{thm:domain}, or note that if one
pictures distributions $P$ of wider and wider support, the minimal
mass bound~(\ref{eq:lower}) implies less and less mass
remains to lower-bound $P(p_0=0)$ and $P_0(p=0)$). To see that the
bound is of a technical rather than essential nature, note that if a
model of bounded-support distributions satisfies (\ref{eq:lower}) and
is uniformly tight, such a constant $\sigma>0$ exists. Consequently,
a sequence of models with growing $\sigma$'s can be used: for given
$P_0=P_{\tht_0,\eta_0}$, there is a lower bound $\bar{\sigma}>0$ such
that the model of theorem~\ref{thm:domain} is well-specified for
all $\sigma>\bar{\sigma}$. So if $\sigma_m\rightarrow\infty$, the
corresponding models $\scrP_m$ are well-specified for large enough
$m$ and the posteriors on those $\scrP_m$ are consistent, \cf\
theorem~\ref{thm:domain}. By diagonalization there exists a sequence
$(\sigma_{m(n)})_{n\geq1}$ that traverses $(\sigma_m)$ slowly enough
in order to guarantee that consistency obtains while we increase
$m(n)$ with the sample size $n$.

To know exactly how slowly we should let $\sigma$ go to infinity,
we use theorem~\ref{thm:barron}: let $\sigma_n$ increase with $n$ and define
$\scrP_n=\{P_{\tht,\eta}\in\scrP:|\tht_1-\tht_2|<\sigma_n,\eta\in H\}$.
Since $N_n=4$ for all $n\geq1$ (namely the sets $V_{+,1}$, $V_{-,1}$,
$V_{+,2}$ and $V_{-,2}$ in the proof of theorem~\ref{thm:domain})
any constant $L>0$ will do, as long as,
\[
  n\,f(\ep/\sigma_n)\rightarrow\infty.
\]
A glance at inequality (\ref{eq:bndPnullPoverQ}) suggests that
condition~(\ref{eq:barron}) applies, if we choose $\Pi$ such that,
\[
  \Pi\bigl( P_{\tht,\eta}\in\scrP : |\tht_1-\tht_2|\geq\sigma_n,\eta\in H \bigr)
  \leq e^{-n\,K},
\]
for some $K>0$. For example, if the family $H$ consists of densities
that display jumps at both $\tht_1$ and $\tht_2$ of some minimal
size $\delta>0$, then $f(x)\geq \ft12\delta\,x$ for values of $x>0$
that are close enough to $x=0$. Consequently, for a model in which
support boundaries represent discontinuous jumps, marginal posterior
consistency obtains if we let $\sigma_n=o(n)$. If $H$
consists of densities that are continuous ($k=0$) or $k\geq1$
times continuously differentiable at the
boundary points, then $f(x)$ is lower-bounded by a multiple of
$x^{k+2}$, which implies that $\sigma_n$ must be of order $o(n^{1/k+2})$.
\end{remark}


\section{Some examples, conclusions and discussion}
\label{sec:disc}

Schwartz's theorem is absolutely central to the frequentist
perspective on Bayesian non-parametric statistics and it has
been in place for more than fifty years: it is beautiful
and powerful, in that it applies to a very wide class of models.
However, its generality with respect to the model implies that
it is rather stringent with respect to the prior. Since choices
for non-parametric priors are usually not abundant, overly
stringent criteria form a problem.

In this paper, an attempt has been made to demonstrate that there is
more flexibility in the criteria for the prior, if one is willing to
accept more strict model conditions. The proposed method applies to
well- and mis-specified models in the form stated, implies
Schwartz's theorem and gives rise to a consistency theorem in
Kullback-Leibler divergence, as well as a metric consistency
theorem for priors that charge metric balls, \eg\ applicable to
the Hellinger metric. Generalizations to sieved models with
Barron's prior mass negligibility condition, to separable
models along the lines of Walker's theorem and to rates of posterior
convergence also fall within the range of the methods proposed.
What remains is to demonstrate practical value based on
applications, to add to examples already given.

\subsection{Some examples}

Because Hellinger consistent density estimation using mixtures is
a well-studied subject, especially with Dirichlet priors, we discuss
that example below in quite some generality, to illustrate
practicality of the proposed methods.
\begin{example}
Consider a model $\scrP$ for observation of one of two real-valued,
dependent random variables $X,Z$, assuming that \emph{if} we would
observe $Z$, the distribution for $X$ would be known: $X|Z=z$
is assumed to have a Lebesgue density $p(\cdot|z):\RR\rightarrow\RR$
such that $z\mapsto p(x|z)$ is bounded and continuous for every $x$.
We observe only an \iid\ sample $X_1,X_2,\ldots$ from $P_0\in\scrP$
and the corresponding $Z_1,Z_2,\ldots$ remain hidden. The model
$\scrP$ then consists of distributions $P_F$ for $X$ with Lebesgue
densities of the form,
\[
  p_{F}(x)=\int_{\RR}p(x|z)\,dF(z)
\]
where the parameter $F$ represents the unknown distribution
of $Z$. For reasons explained below,
assume that $Z\in[0,1]$, so that the space $\scrD$ of all
distributions on $[0,1]$ is compact in the weak topology by
Prokhorov's theorem. Note that for any fixed $x\in\RR$, $F\mapsto p_F(x)$
is weakly continuous. By Scheff\'e's lemma this pointwise continuity
implies weak-to-total-variational continuity of the map
$F\mapsto P_F$, which is equivalent to weak-to-Hellinger
continuity. Since $\scrD$ is weakly compact, this implies that
the model $\scrP$ is Hellinger compact. Note that it is not
necessary for the kernel $z\mapsto p(x|z)$ to be continuous:
for example, models represented by scale-mixtures of uniform
kernels (to represent families of monotone densities), or
exponential densities (to represent densities of completely
monotone distribution functions) also give rise to
weak-to-Hellinger continuous maps $F\mapsto p_F$.

Next we make the additional assumption that the
$L_2$-condition (\ref{eq:elltwoball}) is satisfied;
for example in the well-known normal location mixture model,
where $X|Z=z$ is distributed normally with mean $z$ \cite{Ghosal99},
the family $\scrP=\{p_F:F\in \scrD\}$ is contained in an
envelope that allows straightforward verification
of (\ref{eq:elltwoball}) (for details, see the proof of theorem~3.2
in \cite{Kleijn06}). For mixtures arising from other kernels (like the
scale mixtures of uniform and exponential kernels mentioned already)
condition (\ref{eq:elltwoball}) has to be verified separately and
does not appear to be overly stringent.

With totally boundedness and (\ref{eq:elltwoball}) established,
note that any prior $\Pi$ on $\scrD$ that is Borel for the weak topology
induces a prior that is Borel for the Hellinger topology on the model
$\scrP$. If the weak support of $\Pi$ equals $\scrD$ then the induced
Hellinger support includes $\scrP$. For instance, a Dirichlet prior
for $F$ with base measure of full support on $[0,1]$ will suffice to
conclude from theorem~\ref{thm:hell} that the posterior is Hellinger
consistent,
\[
  \Pi\bigl(\,P\in\scrP\,:\,H(P_F,P_0)>\ep\,\bigm|\,X_1,\ldots,X_n\,\bigr)
  \convas{P_0}0.
\]
for every $\ep>0$. Other priors on $\scrD$, like Gibbs-type measures
of full weak support \cite{DeBlasi13} would also suffice, a result that
is perhaps not so well known. In fact, consistency applies for any
bounded, continuous (and some semi-continuous) kernel(s)
$x\mapsto p(x|z)$ such that mixture densities satisfy (\ref{eq:elltwoball}).
\end{example}
To demonstrate that the approach advocated in this
paper applies where Schwartz's theorem fails, we look at the domain
boundary problem of example~\ref{ex:noKLpriors}.
\begin{example}
\label{ex:fixedwidthdomain}
Assume that the width of the support of $p_0$ is
equal to one. The model consists of densities $\eta$ supported on
$[0,1]$ shifted over $\tht$ in (some subset $D$ of) $\RR$,
\[
  p_{\tht,\eta}(x) = \eta(x-\tht)\,1_{ [\tht,\tht+1] }(x).
\]
Consider $H$ with some prior $\Pi_H$ and a prior $\Pi_{\Tht}$ on
$\Tht=\RR$ with a Lebesgue density that is continuous and strictly
positive on all of $\RR$. 
Note that if $\tht\neq\tht'$ the Kullback-Leibler divergence
of $P_{\tht,\eta}$ with respect to $P_{\tht',\eta'}$ is
infinite, for all $\eta,\eta'\in H$. Hence, for given
$P_0=P_{\tht_0,\eta_0}\in\scrP$, Kullback-Leibler neighbourhoods
do not have any extent in the $\tht$-direction:
\[
  \Bigl\{P_{\tht,\eta}\in\scrP:-P_0\log\frac{dP_{\tht,\eta}}{dP_0}<\ep\Bigr\}
  \subset
  \bigl\{P_{\tht_0,\eta}\in\scrP:\eta\in H\bigr\}.
\]
In order for Schwartz's theorem to apply, a prior satisfying
(\ref{eq:KLprior}) is required: in this example, that requirement
implies that,
\[
  \Pi\bigl(P_{\tht,\eta}:\eta\in H\bigr)>0,
\]
for all $\tht$, which is not possible (unless $D$ is countable).

The construction of example~\ref{ex:domain} remains applicable,
however. In fact, in the present, fixed-width simplification the
situation is more transparent: if we write $P_0=P_{\tht_0,\eta_0}$ and
$V=V_+\cup V_-$ with $V_+=\{P_{\tht,\eta}:\tht>\tht_0+\ep,\eta\in H\}$
and $V_-=\{P_{\tht,\eta}:\tht<\tht_0-\ep,\eta\in H\}$ for some $\ep>0$,
then we choose $B_+=\{P_{\tht,\eta}:\tht_0+\ft12\ep<\tht<\tht_0+\ep,
\eta\in H\}$ and $B_-=\{P_{\tht,\eta}:\tht_0-\ep<\tht<\tht_0-\ft12\ep,
\eta\in H\}$, so that $\Pi(B_{\pm})>0$. Consider only $\al=0$ and
notice that the mismatch in extent of supports implies that,
\[
  P_0(p>0)\leq 1-f(\ep)<1,
\]
for all $P\in\co{V_{\pm}}$, based on (\ref{eq:lower}). If $H$ is
chosen such that for all
$P\in V_{\pm}$, $\sup_{Q\in B_{\pm}}P_0(p/q)<\infty$, then
(\ref{eq:consdomain}) follows (even regardless of the prior on $H$).
Larger spaces $H$ can be considered if the sets $B_{\pm}$ are
restricted appropriately while maintaining $\Pi(B_{\pm})>0$.
Conclude that for the estimation of an unknown $\tht_0\in\RR$,
Schwartz's theorem does not apply, while example~\ref{ex:domain}
remains in effect. 
\end{example}
To show that our proposed approach continues to apply in cases where
the Ghosal-Ghosh-van~der~Vaart theorem does not, we look at the parametric
heavy-tails problem of example~\ref{ex:noGGVpriors}.
\begin{example}
\label{ex:heavytail}
Consider example~\ref{ex:noGGVpriors}: the
sample $X_1, X_2,\ldots$ consists of integers drawn independently
from a distribution $P_a$, $(a\geq1)$, defined by,
\[
  p_a(k) = P_a(X=k) = \frac{1}{Z_a}\frac{1}{k^a(\log k)^3}
\]
for all $k\geq2$, where $Z_a$ is the normalization constant. The
parameter $a$ is smooth and the Fisher information is non-singular,
so $a$ can be estimated at parametric rate, but as noted, there
exists no prior for the parameter $a$ such that condition
(\ref{eq:GGV}) can be satisfied for all $P_0$ in the model.

Corollary~\ref{cor:altGGV} remains valid, however, and demonstrates
that the posterior converges at $\sqrt{n}$-rate. Because
corollary~\ref{cor:altGGV} is formulated for totally-bounded parameter
spaces only, without a negligiblility condition like
(\ref{eq:sievemodelcond}), we restrict the parameter $a$ to a
bounded interval $I$, \ie\ the model is $\scrP=\{P_a:a\in I=[1,L]\}$,
for some $L>1$. (But the result below is expected to hold also without
this restriction.)

For any rate $\ep_n$ that is slower than $n^{-1/2}$, write
$\ep_n=n^{-1/2}M_n$, with $M_n\rightarrow\infty$ and note that we
only have to consider $M_n$ that diverge very slowly, \ie\ $\ep_n$
that are arbitrarily close to the parametric rate. Also note
that there exist constants $M_1,M_2>0$ such that,
\begin{equation}
  \label{eq:bndKL}
  M_1^2(b-a)^2 \leq -P_a\log(p_b/p_a) \leq M_2^2(b-a)^2,
\end{equation}
(because scores have expectation zero and the Fisher information
is non-singular). Define
$V_n=\{P:H(P,P_0)\geq M\ep_n\}$ for some $M>0$. We cover $V_n$
with Hellinger balls
$V_{n,i}$ ($1\leq i\leq N_n$) of radius $\ft12M\ep_n$. Note that
$H(P_b,P_c)\leq M_2|c-b|$ for all $b,c\in I$, so
$N_n=N(\ft12M\ep_n,\scrP,H)\leq N((M/2M_2)\ep_n,I,|\cdot|)\leq
2M_2|I|/(M\ep_n)$. 

Defining also $B_n=\{Q:-P_0\log(dQ/dP_0)<\ep_n^2\}$, we note
that $B_n\subset\{P_b:|b-a|<\ep_n/M_1\}$. Hence, for any
$a\geq1$, any $P_c\in V_n$ and any $P_b$ with $|b-a|<\ep_n/M_2$,
we have,
\[
  P_a\Bigl(\frac{p_c}{p_b}\Bigr)
    = \frac{Z_b}{Z_c}\sum_{k\geq2}\frac{1}{Z_a}\frac{1}{k^a(\log k)^3}
      \frac{k^b}{k^c}
    \leq \frac{Z_b}{Z_c}
\]
because $b<c$ if $M$ is chosen large enough. Since $I$ is compact
and $I\rightarrow\RR:b\mapsto Z_b$ is continuous, $b\mapsto Z_b$ is
bounded, so that for every $P_c\in V_{n,i}$, the integrability condition
$\sup\{P_a(dP_c/dQ): Q\in B_n\}<\infty$ holds. Due to the second
inequality of (\ref{eq:bndKL}), any Borel prior on $I$ of
full support is a KL prior. More specifically, if we choose the
uniform prior on $I$, $\Pi(B_n)\geq \Pi(b\in I:|b-a|\leq \ep_n/M_2)
\geq (|I|M_2)^{-1}\ep_n$. Conclude that theorem~\ref{thm:GGV} does
not apply, but the conditions of corollary~\ref{cor:altGGV} are met
for any rate above $n^{-1/2}$, so the posterior for $a$ converges
at parametric rate.
\end{example}

\subsection{Conclusions and discussion}

It appears safe to conclude that the approach proposed is
versatile where Schwartz's theorem and the Ghosal-Ghosh-van~der~Vaart
theorem are not applicable: the extra flexibility allows that
we `tailor' the prior to the problem using model properties,
rather than being forced to deal with Kullback-Leibler
neighbourhoods or second-order sub-neighbourhoods.

Technically, our proposal is based on the same quantities that
play a central role in \cite{Wong95} (which makes extensive use
of Hellinger transforms to control sieve approximations) and
\cite{Kleijn06}, which applies the minimax theorem in various
ways to prove existence of tests. (In that sense, this work is
more specific than \cite{Kleijn18} which is not limited to
\iid\ data and does not apply the minimax theorem.) 
Regarding the connection with \cite{Kleijn06}, note that there
is a form of mis-specification
\cite{Kleijn04,Kleijn06,Kleijn12b} that applies: $P_0^n$ is
\emph{not equal} to (localised versions of) the prior
predictive distribution, so
frequentist use of Bayesian methodology implies a marginal
distribution for the data that does \emph{not} coincide with
$P_0$. It appears that the asymptotic manifestation of this mismatch
depends on the local prior predictive distributions $P_n^{\Pi|B}$:
if those match $P_0^n$ closely enough (see lemma~\ref{lem:HTBnd}),
testability is maintained and consistency obtains (but see
\cite{Kleijn18} for much more). For analogies
at the conceptual level, compare with Walker's notion of a
\emph{restricted predictive density} in \cite{Walker04} and the
concept of \emph{data-tracking} introduced in \cite{Walker05}, 
relating to Barron's counterexample \cite{Barron99} and that of
Diaconis and Freedman \cite{Diaconis86}: if the prior assigns much
weight to neighbourhoods of $P_0$ that are sensitive to
\emph{data-tracking}, as defined in section~3.3 of \cite{Walker05},
inconsistency of the posterior may occur. From the perspective of
this paper, it appears that data-tracking is controlled sufficiently 
whenever $\pi_{P_0}(W,B)<1$, for a collection of convex sets
$W$ that cover the alternative and a suitably chosen $B$
with $\Pi(B)>0$.


\section*{Acknowledgements}
\label{sec:ack}

The authors would like to thank P. Gr\"unwald for a very useful
suggestion and the referees for this paper for many
helpful comments. BK also thanks Y.~D.~Kim and S.~Petrone
and the {\it Statistics Department of Seoul National University,
South Korea}, and the {\it Dipartimento di Scienze delle Decisioni,
Univesita Bocconi, Milano, Italy} for their kind hospitality. YYZ thanks
the {\it Korteweg-de~Vries Institute} of the {\it University of
Amsterdam} for its hospitality.


\appendix
\section{Some properties of Hellinger transforms}
\label{app:hell}

Given two finite measures $\mu$ and $\nu$,
the Hellinger transform is defined as follows for all $0\leq\al\leq1$:
\[
  \rho_\al(\mu,\nu) = \int \Bigl(\frac{d\mu}{d\sigma}\Bigr)^\al
    \Bigl(\frac{d\nu}{d\sigma}\Bigr)^{1-\al}\,d\sigma,
\]
where $\sigma$ is a $\sigma$-finite measure that dominates both
$\mu$ and $\nu$ (\eg\ $\sigma=\mu+\nu$). 

For $P$ and $Q$ such that $P_0(dP/dQ)<\infty$ define
$d\nu_{P,Q}=(dP/dQ)dP_0$ and note that,
\[
  P_0\Bigl(\frac{dP}{dQ}\Bigr)^\al = \rho_{\al}(\nu_{P,Q},P_0)
    = \rho_{1-\al}(P_0,\nu_{P,Q}).
\]
Properties of the Hellinger transform that are used in the main text
are listed in the following lemma, which extends lemma~6.3 in Kleijn
and van~der~Vaart (2006).
\begin{lemma}
\label{lem:ht}
For a probability measure $P$ and a finite measure $\nu$
(with densities $p$ and $r$ respectively), the function
$\rho:[0,1]\rightarrow\RR:\alpha\mapsto \rho_\alpha(\nu,P)$ is convex
on $[0,1]$ with:
\[
  \rho_\alpha(\nu,P)\rightarrow P(r>0),\quad
    \text{as $\alpha\downarrow 0$},
  \qquad\rho_\alpha(\nu,P)\rightarrow \nu(p>0),\quad
    \text{as $\alpha\uparrow 1$}.
\]
Furthermore, the function $\al\mapsto\rho_\al(\nu,P)$ is continuously
differentiable on $[0,1]$ with derivative,
\[
  \frac{d\rho_\alpha(\nu,P)}{d\alpha}
    = P\,1_{r>0}\,\Bigl(\frac{r}{p}\Bigr)^\al\log(r/p),
\]
(which may be equal to $-\infty$).
\end{lemma}
\begin{proof}
The function $\al\mapsto e^{\al y}$ is convex on $(0,1)$
for all $y\in[-\infty,\infty)$,
implying the convexity of $\al\mapsto \rho_\al(\nu,P)=P(r/p)^\al$
on $(0,1)$. The function $\al\mapsto y^\al=e^{\al\log y}$ is
continuous on $[0,1]$ for any $y>0$, is decreasing for $y<1$, increasing
for $y>1$ and constant for $y=1$. By monotone convergence, as
$\al\downarrow 0$,
\[
  \nu\Bigl(\frac{p}{r}\Bigr)^\al1_{\{0<p<r\}}
  \uparrow \nu\Bigl(\frac{p}{r}\Bigr)^01_{\{0<p<r\}}
  =\nu(0<p<r).
\]
By the dominated convergence theorem (note that $(p/r)^{1/2}1_{\{p\geq r\}}$
upper-bounds $(p/r)^\al1_{\{p\geq r\}}$
for $\al\leq 1/2$) we have,
\[
  \nu\Bigl(\frac{p}{r}\Bigr)^\al1_{\{p\geq r\}}
  \rightarrow \nu\Bigl(\frac{p}{r}\Bigr)^01_{\{p\geq r\}}
  =\nu(p\geq r),
\]
as $\al\downarrow 0$. Combining the two preceding displays, we see that
$\rho_\al(\nu,P)=P(p/r)^\al\rightarrow P(r>0)$ as
$\alpha\downarrow 0$. Upon substitution of $\al$ by $1-\al$, one
finds that $\rho_\al(\nu,P)\rightarrow \nu(p>0)$ as $\al\uparrow1$. 

Let $\al_0\in[0,1]$ be given. By the convexity of $\al\mapsto e^{\al y}$
for all $y\in\RR$, the map
$\al\mapsto f_\al(y)=(e^{\al y}-e^{\al_0 y})/(\al-\al_0)$ decreases to
$y\,e^{\al_0 y}$ as $\al\downarrow\al_0$, and it increases to
$y\,e^{\al_0 y}$ as $\al\uparrow\al_0$. First consider the case that
$\al\geq\al_0$: for $y\leq 0$ we have $f_\al(y)\leq 0$, while for
$y\geq 0$,
\[
  f_\al(y)\leq \sup_{\al_0<\al'\leq\al} y e^{\al' y}\leq y e^{\al y}
    \leq \frac {1}{\ep} e^{(\al+\ep)y},
\]
so that $f_\al(y)\leq 0\vee\ep^{-1}e^{(\al+\ep)y}1_{y\geq 0}$.
Consequently, we have:
\[
  \Bigl(\frac{r}{p}\Bigr)^{\al_0}\frac{e^{(\al-\al_0)\log(r/p)}-1}{\al-\al_0}
    \downarrow\Bigl(\frac{r}{p}\Bigr)^{\al_0}\log\Bigl(\frac{r}{p}\Bigr),
    \quad\text{as $\al\downarrow\al_0$},
\] and is bounded above by
$0\vee \ep^{-1}(r/p)^{\al_0+2\ep}1_{r\geq p}$ for small
$\ep>\al-\al_0>0$, which is $P$-integrable for small enough $\ep$.
We conclude that,
\[
  \frac{1}{\al-\al_0}(\rho_\al(\nu,P)-\rho_{\al_0}(\nu,P))
  \downarrow P\,1_{r>0}\,\Bigl(\frac{r}{p}\Bigr)^{\al_0}
    \log\Bigl(\frac{r}{p}\Bigr),\quad\text{as $\al\downarrow\al_0$},
\]
by monotone convergence. For $\al<\al_0$ a similar argument can be given.
Convexity of $\al\mapsto P\,1_{r>0}\,(r/p)^\al\log(r/p)$ implies
continuity of the derivative.
\end{proof}

\section{Proofs}
\label{app:proofs}

This section contains all proofs of theorems and lemmata in the main text,
as well as some remarks and side-notes.

\subsection{Proofs for section~\ref{sec:cons}}

\begin{proof}{\it (Proposition~\ref{prop:dompriorpred})}\\
For any $A\in\sigma_n:=\sigma(X_1,\ldots,X_n)$ and any
model subset $U'$ such that $\Pi(U')>0$,
\[
  P_0^n(A) \leq \int P^n(A)\,d\Pi(P|U')+\sup_{P\in U'} |P^n(A)-P_0^n(A)|.
\]
Now assume that $A$ is a null-set of $P_n^{\Pi}$; since $\Pi(U')>0$,
$\int P^n(A)\,d\Pi(P|U')=0$. For some $\ep>0$, take
$U'=\{P:|P^n(A)-P_0^n(A)|<\ep\}$, note that $U'$
contains a total-variational neighbourhood and therefore
a Hellinger neighbourhood, to conclude that
$P_0^n(A)<\ep$ for all $\ep>0$.
\end{proof}
\begin{remark}
Sets of the form $U'$ in the proof of
proposition~\ref{prop:dompriorpred} form a sub-basis for a
topology $\scrT_n$ that is weaker than the Hellinger
topology. So the condition of proposition~\ref{prop:dompriorpred} can be
weakened to ``all $\scrT_n$-neighbourhoods of $P_0$''.
(See remark~3.6\,(2) in \cite{Strasser85}.)
\end{remark}
\begin{proof}{\it (Corollary~\ref{cor:KLpriorpred})}\\
Since $-P_0\log(p/p_0)\geq H(P_0,P)$, every Hellinger ball contains a
Kullback-Leibler-ball. So (\ref{eq:KLprior}) implies that $\Pi(U)>0$
for every Hellinger ball $U$. According to
proposition~\ref{prop:dompriorpred}, this implies (\ref{eq:dompriorpred}).
\end{proof}
\begin{proof}{\it (Lemma~\ref{lem:cons})}\\
For a set $V$ covered by measurable $V_1,\ldots,V_N$, almost-sure 
convergence per individual $V_i$ implies the assertion. So we fix
some $1\leq i\leq N$ and note that,
\[
  P_0^n\Pi(V_i|X_1,\ldots,X_n)
  \leq P_0^n\phi_{i,n} + P_0^n\Pi(V_i|X_1,\ldots,X_n)(1-\phi_{i,n}).
\]
By Fubini's theorem,
\begin{equation}
  \label{eq:pivi}
  \begin{split}
  P_0^n\Pi(&V_i|X_1,\ldots,X_n)(1-\phi_{i,n})
    = P_0^n\int_{V_i}\frac{dP^n}{dP_n^{\Pi}}(1-\phi_{i,n})\,d\Pi(P)\\
    &\leq \Pi(V_i)\,\sup_{P\in V_i} P_0^n \frac{dP^n}{dP_n^\Pi}(1-\phi_{i,n}).
  \end{split}
\end{equation}
From (\ref{eq:testcond}) we conclude that
$P_0^n\Pi(V_i|X_1,\ldots,X_n)\leq e^{-nD_i}$, for
large enough $n$. Apply Markov's inequality to find that,
\[
  P_0^n\bigl(\, \Pi(V_i|X_1,\ldots,X_n) \geq e^{-\ft{n}2D_i}\, \bigr)
    \leq e^{-\ft{n}2D_i},
\]
so that the first Borel-Cantelli lemma guarantees,
\[
  P_0^\infty\Bigl(\, \limsup_{n\rightarrow\infty}
    \bigl( \Pi(V_i|X_1,\ldots,X_n) - e^{-\ft{n}2D_i} \bigr) >0 \,\Bigr) = 0.
\]
Replicating this argument for all $1\leq i\leq N$, assertion
(\ref{eq:aszero}) follows.
\end{proof}
\begin{remark}
\label{rem:pivi}
In going from (\ref{eq:pivi})
to the next bound, a factor $\Pi(V_i)$ was (upper bounded by one and)
omitted. This factor does not play a role in the proof of
theorem~\ref{thm:main}, but it is crucial in
subsection~\ref{sub:walker}.
\end{remark}
\begin{remark} The exponential bound for testing power is sufficient for
$P_0$-almost-sure convergence of the posterior for $V$,
\cf\ (\ref{eq:aszero}). If only convergence in $P_0$-probability
is required, the right-hand side of (\ref{eq:testcond}) may be relaxed
to be of order $o(1)$.
\end{remark}
\begin{example}
\label{ex:lecamdim}
Suppose that we wish to prove consistency relative to
some metric $d$ on $\scrP$ but coverings of the model by $d$-balls are
not finite. Then we may try the following construction: for $\ep>0$,
we define $W=\{P\in\scrP:d(P,P_0)>\ep\}$ and $W_k=\{P\in\scrP:2^{k-1}\,\ep
\leq d(P,P_0)<2^k\,\ep\}$, ($k\geq1$). Assume that the covering numbers
$N_k:=N_k(2^{k-2}\ep,W_k,d)$ of the model subsets $W_k$ (related to the
so-called \emph{Le~Cam dimension} of the model \cite{LeCam73}) are
finite. Let $V_{k,1},\ldots,V_{k,N_k}$ be $d$-balls of radius $2^{k-2}\ep$
covering $W_k$. Assume that for every $d$-ball
$V_{k,i}$, ($1\leq i\leq N_k$), there exists a test sequence
$(\phi_{k,i,n})_{n\geq1}$ such that (\ref{eq:testcond}) is satisfied
with $D_{k,i}\geq d^2(V_{k,i},P_0)$. Then, for every $n\geq1$,
\[
  \begin{split}
  P_0^n\Pi(&W|X_1,\ldots,X_n)\\[1ex]
  &\leq \sum_{k\geq1} \sum_{1\leq i\leq N_k} P_0^n\Pi(V_{k,i}|X_1,\ldots,X_n)
  \leq \sum_{k\geq1} N_k\,e^{-2^{2k-4}\,n\,\ep^2}.
  \end{split}
\]
If we show that the right-hand side goes to zero as
$n\rightarrow\infty$, the posterior is $d$-consistent.
\end{example}
\begin{proof}{\it (Lemma~\ref{lem:ntests})}\\
We appeal to an argument from \cite{Kleijn06} based on the minimax
theorem (see, \eg, theorem~45.8 in Strasser (1985) \cite{Strasser85}).
According to lemma~6.1 of \cite{Kleijn06} there exists a
test $(\phi_n)$ that minimizes the \lhs\ of (\ref{eq:testpower}) and,
\[
  \sup_{P\in V} \Bigl( P_0^n\phi_n +
   P_0^n \frac{dP^n}{dP_n^\Pi}(1-\phi_n) \Bigr)
    \leq \sup_{P_n\in\co{V^n}} \vinf_{\phantom{|}\phi\phantom{|}} \Bigl( P_0^n\phi +
      P_0^n\frac{dP_n}{dP_n^\Pi}(1-\phi) \Bigr).
\]
The infimal $\phi$ on the right-hand side may now be chosen
specifically tuned to $P_n$, and equals the indicator
$\phi=1_{\{ dP_n/dP_n^\Pi > 1\}}$. For any $\al\in[0,1]$,
\[
  \int 1_{\{ dP_n/dP_n^\Pi > 1\}}\,dP_0^n +
      \int \frac{dP_n}{dP_n^\Pi}\,1_{\{ dP_n/dP_n^\Pi \leq 1\}}\,dP_0^n
  \leq
  \int \Bigl(\frac{dP_n}{dP_n^\Pi}\Bigr)^\al\,dP_0^n,
\]
which enables an upper-bound for testing power,
\[
  \begin{split}
    \sup_{P_n\in\co{V^n}} \Bigl( P_0^n\,&1_{\{ dP_n/dP_n^\Pi > 1\}} +
      P_0^n \frac{dP_n}{dP_n^\Pi}\,1_{\{ dP_n/dP_n^\Pi \leq
      1\}} \Bigr)\\
    &\leq \sup_{P_n\in\co{ V^n}}\vinf_{\phantom{|}0\leq\al\leq1\phantom{|}}
      P_0^n\Bigl(\frac{dP_n}{dP_n^\Pi}\Bigr)^\al,
  \end{split}
\]
 in terms of the Hellinger transform.
\end{proof}
\begin{remark}
For some background on Hellinger transforms, see section~3.6 of
\cite{LeCam86} as well as the lemma of appendix~\ref{app:hell}.
Hellinger transforms occur naturally in minimax problems: compare,
for example, with remark~2 in section~16.4 of \cite{LeCam86}, or
with inequality~(6.39) in \cite{Kleijn06}.
\end{remark}
The proofs of lemmata~\ref{lem:cons} and
\ref{lem:ntests} do not depend on \iid-ness of the data and could have
been formulated in the same form for models of dependent
or non-identically distributed samples. The following lemma \emph{does}
rely on the \iid\ property of the sample and reduces the testing
criterion to an $n$-independent condition, \cf\ (\ref{eq:power}).\par
\begin{proof}{\it (Lemma~\ref{lem:HTBnd})}\\
Let $0\leq\al\leq1$ be given. Note that for all $n\geq1$, 
$P_n^\Pi(A)\geq \Pi(B)\,P_n^{\Pi|B}(A)$ for all $A\in\sigma(X_1,\ldots,X_n)$.
Combining that with the convexity of $x\mapsto x^{-\al}$ on $(0,\infty)$,
we see that,
\begin{equation}
  \label{eq:PQ}
  P_0^n\Bigl(\frac{dP_n}{dP_n^\Pi}\Bigr)^\al
    \leq \Pi(B)^{-\al}\,P_0^n\Bigl(\frac{dP_n}{dP_n^{\Pi|B}}\Bigr)^\al
    \leq \Pi(B)^{-\al}\,P_0^n\int \Bigl(\frac{dP_n}{dQ^n}\Bigr)^{\al}
      \,d\Pi(Q|B).
\end{equation}
With the use of Fubini's theorem and lemma~6.2 in Kleijn and
van~der~Vaart (2006) \cite{Kleijn06} which
says that Hellinger transforms factorize when taken over convex hulls
of products, we find:
\[
  \begin{split}
  \sup_{P_n\in\co{V^n}}\vinf_{\phantom{|}0\leq\al\leq1\phantom{|}}\Pi(B)^{-\al}
    &\int P_0^n\Bigl(\frac{dP_n}{dQ^n}\Bigr)^{\al} d\Pi(Q|B)\\
    &\leq \inf_{0\leq\al\leq1}\Pi(B)^{-\al}\int \sup_{P_n\in\co{V^n}}
      P_0^n\Bigl(\frac{dP_n}{dQ^n}\Bigr)^{\al} d\Pi(Q|B)\\
    &\leq \inf_{0\leq\al\leq1}\Pi(B)^{-\al}\int\Bigl[\sup_{P\in \co{V}}
      P_0\Bigl(\frac{dP}{dQ}\Bigr)^{\al}\Bigr]^n d\Pi(Q|B).
  \end{split}
\]
Applying (\ref{eq:PQ}) with $\al=1$, $P_n=P^n$, and using that 
for all $P\in V$, $\sup_{Q\in B} P_0(dP/dQ)<\infty$, we see that also
$P_0^n({dP^n}/{dP_n^\Pi})<\infty$. By (\ref{eq:testpower}), we obtain
(\ref{eq:iidpower}).
\end{proof}

\subsection{Proofs for section~\ref{sec:variations}}

\begin{proof}{\it (Lemma~\ref{lem:KLequiv})}\\
Assume that (\ref{eq:KLsep}) holds. Lemma~\ref{lem:ht} says that
$\al\mapsto P_0(dP/dQ)^\al$ is convex and
continuously differentiable on $(0,a)$. So for all $\al\in(0,a)$,
\begin{equation}
  \label{eq:abscont}
  \sup_{Q\in B}\sup_{P\in W} P_0\Bigl(\frac{dP}{dQ}\Bigr)^\al
    \leq 1 + \al\,\sup_{Q\in B} \sup_{P\in W}
     P_0\Bigl(\frac{dP}{dQ}\Bigr)^{\al}\log\frac{dP}{dQ}.
\end{equation}
The function
\[
  \al\mapsto \sup_{Q\in B}\sup_{P\in W}
    P_0\Bigl(\frac{dP}{dQ}\Bigr)^{\al}\log\frac{dP}{dQ},
\]
is convex (hence continuous on $(0,a)$ and upper-semi-continuous
at $0$) and, due to (\ref{eq:KLsep}), strictly negative at $\al=0$.
As a consequence, there exists an interval $[0,\al_0]$ on which the
function in the above display is strictly negative. Based on
(\ref{eq:abscont}) there exists an $\al_0\in[0,1]$ such that
$\sup_{P,Q}P_0({dP/dQ})^{\al_0}<1$ and we conclude that (\ref{eq:powerbnd})
holds. Conversely, assume that (\ref{eq:KLsep}) does not hold. Let
$P\in W$, $Q\in B$ and $\al\in[0,1]$ be given; by Jensen's inequality,
\[
  P_0\Bigl(\frac{dP}{dQ}\Bigr)^\al
    \geq \exp\Bigl(\al P_0\log\frac{dP}{dQ}\Bigr)
    = \exp\biggl( \al \Bigl( P_0\log\frac{dP}{dP_0}
        - P_0\log\frac{dQ}{dP_0}\Bigr)\biggr).
\]
Therefore,
\[
  \sup_{Q\in B}\sup_{P\in W}P_0\Bigl(\frac{dP}{dQ}\Bigr)^\al
  \geq 
    \exp\Bigl(\al\sup_{Q\in B} -P_0\log\frac{dQ}{dP_0}\Bigr)
    \exp\Bigl(-\al\inf_{P\in W} -P_0\log\frac{dP}{dP_0}\Bigr),
\]
which is greater than or equal to one for all $\al\in[0,1]$.
\end{proof}
\begin{proof}{\it (Theorem~\ref{thm:schwartz})}\\
Let $\ep>0$ be given. If one covers the complement of a
Hellinger ball $V$ of radius $2\ep$ centred on $P_0$ by a finite,
convex cover of Hellinger balls $V_i$ ($1\leq i\leq N$) of radii $\ep$,
then for every $P\in V_i$, $-P_0\log(dP/dP_0)\geq H(P,P_0) >
\ep$. Together with the upper bound on the Kullback-Leibler-divergence of
$Q$ with respect to $P_0$ for all $Q\in B$ formulated by (\ref{eq:KLprior}),
this implies that (\ref{eq:KLsep}) is satisfied.
Under the assumption that there exists a Kullback-Leibler neighbourhood
$B$ such that $\sup_{Q\in B}P_0(dP/dQ)<\infty$ for all $P\in V$ and
(\ref{eq:KLprior}), lemma~\ref{lem:HTBnd} guarantees the existence of
tests that satisfy the testing condition of lemma~\ref{lem:cons}.
Because every total-variational ball contains a Kullback-Leibler
neighbourhood, proposition~\ref{prop:dompriorpred} asserts that
(\ref{eq:dompriorpred}) is satisfied and we conclude that 
the posterior satisfies (\ref{eq:Schwartzconsistency}).
\end{proof}
\begin{proof}{\it (Theorem~\ref{thm:KLDcons})}\\
For every $1\leq i\leq N$, there exists a constant $b_i>0$ such that
for every $P\in W_i:=\co{V_i}$, $-P_0\log(dP/dP_0) > b_i$.
Denoting the Kullback-Leibler radius of $B$ by $b>0$, we define
$B_i=\{P\in\scrP:-P_0\log(dP/dP_0)<b_i\wedge b\}$ to satisfy
(\ref{eq:KLsep}). Note that, by assumption, $\Pi(B_i)>0$ and
$\sup_{Q\in B_i}P_0(dP/dQ)<\infty$ for all $P\in\scrP$.
Every total-variational neighbourhood of $P_0$ contains a
Kullback-Leibler neighbourhood, so combination of lemma~\ref{lem:KLequiv},
lemma~\ref{lem:HTBnd} and theorem~\ref{thm:main} proves posterior
consistency in Kullback-Leibler divergence \cf\ (\ref{eq:KLDcons}).
\end{proof}
\begin{proof}{\it (Theorem~\ref{thm:hell})}\\
Proposition~\ref{prop:dompriorpred} guarantees that $P_0^n \ll P_{n}^{\Pi}$,
for all $n\geq1$.
For given $\ep>0$, let $V$ denote $\{P\in\scrP:H(P,P_0)>2\ep\}$.
Since $\scrP$ is totally bounded in the Hellinger metric, there exist
$P_1,\ldots,P_N$ such that the model subsets 
$V_i=\{P\in\scrP:H(P,P_i)<\ep\}$ form a cover of $V$. On the basis of
the constant $L$ of (\ref{eq:elltwoball}), define
$B=\{Q\in\scrP:H(Q,P_0)<\ep^2/(4L)\wedge\ep'\}$, where $\ep'$ is the
Hellinger radius of $B'$. Since Hellinger balls are
convex, we have for all $1\leq i\leq N$,
\[
  \sup_{P\in \co{V_i}} \sup_{Q\in B} P_0\Bigl(\frac{p}{q}\Bigr)^{1/2}
    \leq 1-\ft14\ep^2 \leq e^{-\ft14\ep^2}.
\]
By the Cauchy-Schwarz inequality, for every $P\in V$,
\[
  \sup_{Q\in B} P_0\Bigl(\frac{p}{q}\Bigr)\leq \sup_{Q\in B}
    \Bigl\|\frac{p_0}{q}\Bigr\|_{2,Q}\,
    \Bigl\|\frac{p}{q}\Bigr\|_{2,Q} < L^2 <\infty.
\]
According to lemmata~\ref{lem:HTBnd} and~\ref{lem:cons},
the posterior is consistent.
\end{proof}
\begin{proof}{\it (Theorem~\ref{thm:drcons})}\\
Reasoning like in the introduction of subsection~\ref{sub:hellpriors},
but now with H\"older's inequality, one finds,
\[
  P_0\Bigl(\frac{p}{q}\Bigr)^{1/r}
    \leq \rho_{1/r}(P,P_0)
      +d_r(P_0,Q)\,\bigl(P_0(p/q)^{s/r}\bigr)^{1/s}
\]
Let $\ep>0$ be given and let $V$ be the complement of a $d_r$-ball of
radius $2\ep$. Cover $V$ by $N$ $d_r$-balls $V_1,\ldots,V_N$ of radii
$\ep$ (which are convex) and note that for all $1\leq i\leq N$ and
$P\in V_i$, $d_r(P,P_0)\geq\ep$. It is shown in the corollary of
theorem~1 of \cite{Toussiant72} that,
\[
  \rho_{1/r}(P_0,P)\leq\Bigl(\frac{2(r-1)}{r}\Bigr)^{1/2}
    \bigl( 1-\ft14\,d_r(P,P_0)^{2r}\bigr)^{1/2},
\]
and with $K=(2(r-1)/r)^{1/2}$, it follows that, 
\[
  P_0\Bigl(\frac{p}{q}\Bigr)^{1/r}
    \leq K(1-\ft14\,\ep^{2r})^{1/2} + L^{s/r\wedge1}\,d_r(P_0,Q).
\]
Since $(1-x)^{1/2}\leq1-\ft12x$ for $x\in(0,1)$, the choice
$\delta=(K/16)L^{-(s/r\wedge1)}\ep^{2r}$ in (\ref{eq:drprior})
guarantees that $P_0(p/q)^{1/r}\leq K(1-(1/16)\ep^{2r})$ for all
$1\leq i\leq N$, $P\in V_i$ and $Q\in B$. If $s\geq r$, Jensen's
inequality implies that $\sup_{Q\in B}P_0(p/q)<\infty$; if $s<r$,
$\sup_{Q\in B}P_0(p/q)<\infty$ by assumption. 
According to lemma~\ref{lem:HTBnd} and
lemma~\ref{lem:cons}, the posterior is consistent.
\end{proof}
\begin{proof}{\it (Lemma~\ref{lem:KLcont})}\\
Continuity implies that every Kullback-Leibler ball around $P_0$
contains an open neighbourhood of $P_0$.
\end{proof}

\subsection{Proofs for section~\ref{sec:sep}}

\begin{proof}{\it (Theorem~\ref{thm:barron})}\\
For given $V$ and $n\geq1$, denote the cover of condition~{\it(i.)} by
$V_1,\ldots,V_{N_n}$ with tests $\phi_{i,n}$, $1\leq i\leq N_n$. Define
$\psi_n=\max_i\phi_{i,n}$ and decompose the $n$-th posterior for $V$
as follows,
\[
  \begin{split}
  P_0^n \Pi(&V|X_1,\ldots,X_n) \leq P_0^n\psi_n+\\[1ex]
    & P_0^n\Pi(V\cap\scrP_n|X_1,\ldots,X_n)(1-\psi_n)
    + P_0^n\Pi(V\setminus\scrP_n|X_1,\ldots,X_n).
  \end{split}
\]
The first term is upper bounded geometrically, $P_0^n\psi_n\leq
\sum_{i=1}^{N_n}P_0^n\phi_{i,n}\leq N_n\exp(-nL)\leq\exp(-\ft12nL)$,
as is the second term, namely,
\[
  \begin{split}
  P_0^n\Pi(V\cap\scrP_n&|X_1,\ldots,X_n)(1-\psi_n)\\
    &\leq\sum_{i=1}^{N_n}P_0^n\Pi(V_i\cap\scrP_n|X_1,\ldots,X_n)(1-\psi_n)\\
    &\leq\sum_{i=1}^{N_n}P_0^n\Pi(V_i\cap\scrP_n|X_1,\ldots,X_n)(1-\phi_{i,n})\\
    &\leq\sum_{i=1}^{N_n}\sup_{P\in V_i}P_0^n \frac{dP^n}{dP_n^\Pi}(1-\phi_{i,n})
    \leq N_n\,e^{-nL} \leq e^{-\ft12nL},
  \end{split}
\]
where we have followed the steps in the proof of theorem~\ref{thm:main}.
Using again the local prior predictive distribution $P_n^{\Pi|B}$
of (\ref{eq:locpriorpred}), the third term satisfies,
\[
  \begin{split}
  P_0^n\Pi(&V\setminus\scrP_n|X_1,\ldots,X_n) 
    \leq P_0^n\Pi(\scrP\setminus\scrP_n|X_1,\ldots,X_n)\\[1ex]
    &= \int_{\scrP\setminus\scrP_n} P_0^n\Bigl(\frac{dP^n}{dP_n^\Pi}\Bigr)d\Pi(P)
    \leq \frac1{\Pi(B)}\int_{\scrP\setminus\scrP_n}
      P_0^n\Bigl(\frac{dP^n}{dP_n^{\Pi|B}}\Bigr)d\Pi(P)\\
    &\leq \frac{\Pi(\scrP\setminus\scrP_n)}{\Pi(B)}
      \sup_{P\in \scrP\setminus\scrP_n}\sup_{Q\in B}\bigl[P_0(p/q)\bigr]^n
    \leq \Pi(B)^{-1} e^{-\ft12Kn}.
  \end{split}
\]
Like at the end of the proof of lemma~\ref{lem:cons}, an application of the
Borel-Cantelli proves the assertion.
\end{proof}
\begin{remark}
A generalization of condition {\it (ii.)} in theorem~\ref{thm:barron}
concerns $n$-dependence in the choice for $B$. Clearly, balancing
of exponential factors then involves an exponential lower bound for
the sequence $(\Pi(B_n))$ as well. However, this remark and other points
of generalization or flexibility more naturally fit into a discussion
of posterior rates of convergence \cite{Barron99,Ghosal00,KleijnXX}.
\end{remark}
\begin{proof}{\it (Theorem~\ref{thm:sep})}\\
By monotone convergence,
\[
  \begin{split}
  P_0^n &\Pi(V|X_1,\ldots,X_n)\\
  &\leq P_0^n \Pi\bigl(\,\cup_{i\geq1}V_i\bigm|X_1,\ldots,X_n\bigr)
  \leq \sum_{i\geq1} P_0^n \Pi(V_i|X_1,\ldots,X_n).
  \end{split}
\]
We treat the terms
in the sum separately with the help of test sequences $(\phi_{i,n})$,
for all $i\geq1$, following the proof of lemma~\ref{lem:cons}:
\begin{equation}
  \label{eq:sumvis}
  P_0^n \Pi(V|X_1,\ldots,X_n)
 \leq \sum_{i\geq1} \Bigl(
    P_0^n\phi_{i,n} + 
    \Pi(V_i)\,\sup_{P\in{V_i}} P_0^n\frac{dP^n}{dP_n^\Pi}(1-\phi_{i,n})\Bigr).
\end{equation}
(Note that, here, we maintain the factor $\Pi(V_i)$ of
remark~\ref{rem:pivi}, for reasons that will become clear shortly.)
Like in the proof of lemma~\ref{lem:HTBnd}, the assumptions that
$\sup_{Q\in B_i} P_0(dP/dQ)<\infty$ and $\Pi(B_i)>0$, imply that
$P_0^n({dP^n}/{dP_n^\Pi})<\infty$, for all $P\in V_i$. So $\phi_{i,n}$
can be chosen in such a way that,
\begin{equation}
  \label{eq:mmbnd}
  \begin{split}
  P_0^n\phi_{i,n} &+ 
    \Pi(V_i)\,\sup_{P\in{V_i}} P_0^n\frac{dP^n}{dP_n^\Pi}(1-\phi_{i,n})\\
  &= \sup_{P_n\in\co{V_i^n}}\vinf_{\phantom{|}\phi\phantom{|}} \Bigl(
    P_0^n\phi + \Pi(V_i)\,P_0^n\frac{dP_n}{dP_n^\Pi}(1-\phi)\Bigr)
  \end{split}
\end{equation}
by the minimax theorem. To minimize the \rhs, choose $\phi$ as follows,
\[
  \phi(X_1,\ldots,X_n) 
    = 1\Bigl\{(X_1,\ldots,X_n)\in\scrX^n\,:\,
      \Pi(V_i)\,\frac{dP^n}{dP^{\Pi}_n}(X_1,\ldots,X_n)>1\Bigr\},
\]
and follow the proof of lemma~\ref{lem:HTBnd} to conclude that the
\rhs\ of (\ref{eq:mmbnd}) is upper bounded by,
\[
  \inf_{0\leq\al\leq1} \frac{\Pi(V_i)^\al}{\Pi(B_i)^\al}
    \Bigl[
      \sup_{P\in\co{V_i}}\sup_{Q\in B_i}P_0\Bigl(\frac{dP}{dQ}\Bigr)^\al
    \Bigr]^n.
\]
Combine with (\ref{eq:sumvis}) to arrive at the assertion.
\end{proof}
\begin{remark}
If the relevant metric is
the Hellinger metric, separability of the model has a remarkable
equivalent formulation (see sections~4 and~21 in \cite{Strasser85}):
a collection $\scrP$ of probability measures $P:\scrA\to[0,1]$
is Hellinger separable, if and only if, $\scrP$ is dominated by
a $\sigma$-finite measure and $\scrA$ is a countably generated
$\sigma$-algebra.
\end{remark}
\begin{proof}{\it (Corollary~\ref{cor:varwalker})}\\
Fix $\al=1/2$ and $B_i=B$ in (\ref{eq:prewalker}) and use
(\ref{eq:tstpwr}) to arrive at,
\[
  \begin{split}
  P_0^n&\Pi(V|X_1,\ldots,X_n)\\
    &\leq \Pi(B)^{-1/2}
      \sum_{i\geq1} \Pi(V_i)^{1/2}
      \Bigl[
        \sup_{P\in\co{V_i}}\sup_{Q\in B}P_0\Bigl(\frac{dP}{dQ}\Bigr)^{1/2}
      \Bigr]^n\\
    &\leq \Pi(B)^{-1/2} (1-\gamma)^n \sum_{i\geq1} \Pi(V_i)^{1/2},
  \end{split}
\]
(for some constant $0<\gamma<1$), which goes to zero at geometric rate
if (\ref{eq:walkercond}) holds. This implies that
$\Pi(V|X_1,\ldots,X_n)\convas{P_0}0$.
\end{proof}
\begin{proof}{\it (Corollary~\ref{cor:onlysum})}\\
Given $\ep>0$, define $V=\{P:H(P,P_0)\geq\ep\}$ and let $\{V_i:i\geq1\}$
denote a countable collection of Hellinger balls of radius $\ft14\ep$
with centre points in $V$ that cover $V$, so that,
\begin{equation}
  \label{eq:unifH}
  \inf_{i\geq1} \inf_{P\in\co{V_i}} H(P,P_0) \geq \ft34\ep.
\end{equation}
Inspection of the proof of lemma~\ref{lem:KLequiv} reveals that it
generalizes to the statement that:
\[
  \vinf_{\phantom{|}0\leq\al\leq1\phantom{|}}
    \sup_{i\geq1}
      \sup_{Q\in B}
        \sup_{P\in\co{V_i}}P_0\Bigl(\frac{dP}{dQ}\Bigr)^{\al}<1,
\]
if and only if, $B$ and the $\co{V_i}$ separate in
Kullback-Leibler divergence in the following way,
\[
  \sup_{Q\in B}-P_0\log\frac{dQ}{dP_0}
    <
  \vinf_{\phantom{|}i\geq1\phantom{|}} \inf_{P\in\co{V_i}}-P_0\log\frac{dP}{dP_0},
\]
Note that (\ref{eq:unifH}) serves as a lower bound for the \rhs\ of
the previous display, which enables the choice
$B=\{P\in\scrP:-P_0\log(p/p_0)<\ep/{4}\}$
to guarantee that there exist constants $0<\al',\gamma<1$ such that,
\[
  P_0^n\Pi(V|X_1,\ldots,X_n)
    \leq \Pi(B)^{-\al'} (1-\gamma)^{n\al'} \sum_{i\geq1}\Pi(V_i)^{\al'},
\]
which goes to zero since $\Pi(B)>0$ and the sum is finite by assumption.
\end{proof}

\subsection{Proofs for section~\ref{sec:boundary}}

\begin{proof}{\it (Theorem~\ref{thm:domain})}\\
Let $\ep>0$ be given and consider the (equivalent) metric
$g:\Tht\times\Tht\rightarrow[0,\infty)$ defined by $g(\tht,\tht')=
\max\{|\tht_1-\tht'_1|,|\tht_2-\tht'_2|\}$. Define
$V=\{P_{\tht,\eta}\in\scrP:g(\tht,\tht')>\ep\}$. Concentrate on the
cases $\al=0+$ and $\al=1-$; pick $0<\delta<f(\ep/\sigma)/(2K)$ and
define $B$ as above. Lemma~\ref{lem:ht} says that for all $P\in V$
and $Q\in B$,
\[
  \begin{split}
  P_0\Bigl(\frac{dP}{dQ}\Bigr)^{0+}&=P_0(p>0),\\
  P_0\Bigl(\frac{dP}{dQ}\Bigr)^{1-}&=
    \int{\frac{dP_0}{dQ}}\,1_{\{p_0>0,p>0,q>0\}}\,dP\\
    &\leq P(p_0>0) + \int \Bigl|\frac{dP_0}{dQ}-1\Bigr|1_{\{q>0\}}\,dP\\
    &\leq P(p_0>0) + \Bigl\|\frac{dP_0}{dQ}-1\Bigr\|_{s,Q}\,
      \Bigl\|\frac{dP}{dQ}\Bigr\|_{r,Q}
    < P(p_0>0) + \ft12 f\bigl(\ft{\ep}{\sigma}\bigr),
  \end{split}
\]
by H\"older's inequality. Note that every total-variational neighbourhood
of $P_0$ contains a model subset of the form $B$ and, by assumption,
$\Pi(B)>0$, so that proposition~\ref{prop:dompriorpred} guarantees that
$P_0^n \ll P_{n}^{\Pi}$ for all $n\geq1$. For all $P\in V$,
$\sup_{Q\in B} P_0(dP/dQ)\leq 1+\ft12 f(\ep/\sigma)<\infty$ and for
all $Q\in B$, we have,
\[
  \inf_{0\leq\al\leq1} P_0\Bigl(\frac{dP}{dQ}\Bigr)^{\al}
  \leq \min\bigl\{P_0(p>0),P(p_0>0)\bigr\}+\ft12 f\bigl(\ft{\ep}{\sigma}\bigr),
\]
as an upper bound for testing power.

Identify $P_0$ and $P$ with parameters $(\tht_0,\eta_0)$
and $(\tht,\eta)$, writing $P_0=P_{(\tht_{0,1},\tht_{0,2}),\eta_0}$ and
$P=P_{(\tht_1,\tht_2),\eta}$. By definition of $V$, the support
intervals for $p$ and $p_0$ are disjoint by an interval of
length greater than or equal to $\ep$. Cover $V$ by four sets,
$V_{+,1}=\{P_{\tht,\eta}:\tht_1\geq\tht_{0,1}+\ep,\eta\in H\}$,
$V_{-,1}=\{P_{\tht,\eta}:\tht_1\leq\tht_{0,1}-\ep,\eta\in H\}$,
$V_{+,2}=\{P_{\tht,\eta}:\tht_2\geq\tht_{0,2}+\ep,\eta\in H\}$ and
$V_{-,2}=\{P_{\tht,\eta}:\tht_2\leq\tht_{0,2}-\ep,\eta\in H\}$.
For $P\in \co{V_{+,1}}$, we have,
\[
  \begin{split}
  P_0(p=0)&\geq\int_{\tht_{0,1}}^{\tht_{0,1}+\ep}p_0(x)\,dx
  = \int_{\tht_{0,1}}^{\tht_{0,1}+\ep}\frac1{\tht_{0,2}-\tht_{0,1}}
    \eta_0\Bigl(\frac{x-\tht_{0,1}}{\tht_{0,2}-\tht_{0,1}}\Bigr)dx\\
  &= \int_0^{\ep/(\tht_{0,2}-\tht_{0,1})}\eta_0(z)\,dz
  \geq \int_0^{\ft{\ep}{\sigma}}\eta_0(z)\,dz\geq f\bigl(\ft{\ep}{\sigma}\bigr),
  \end{split}
\]
using (\ref{eq:lower}).
For $P\in\co{V_{-,1}}$, with some $I\geq1$ write
$P=\sum_{i=1}^I\lambda_i\,P_i$ with $\sum_{i=1}^I\lambda_i=1$
and $\lambda_i\geq0$, $P_i=P_{\tht_i,\eta_i}$ for
$\tht_i=(\tht_{i,1},\tht_{i,2})$ with $\tht_{i,1}\leq\tht_{0,1}-\ep$
and $\eta_i\in H$, for all $1\leq i\leq I$. Note that,
\[
  \begin{split}
  P(p_0=0) &= \sum_{i=1}^I\lambda_i\,P_i(p_0=0)
  \geq \sum_{i=1}^I\lambda_i\int_{\tht_{i,1}}^{\tht_{i,1}+\ep}p_i(x)\,dx\\
  &= \sum_{i=1}^I\lambda_i\int_{\tht_{i,1}}^{\tht_{i,1}+\ep}
    \frac1{\tht_{i,2}-\tht_{i,1}}
    \eta_i\Bigl(\frac{x-\tht_{i,1}}{\tht_{i,2}-\tht_{i,1}}\Bigr)dx\\
  &= \sum_{i=1}^I\lambda_i \int_0^{\ep/(\tht_{i,2}-\tht_{i,1})}\eta_i(z)\,dz
  \geq \sum_{i=1}^I\lambda_i \int_0^{\ft{\ep}{\sigma}}\eta_i(z)\,dz
  \geq f\bigl(\ft{\ep}{\sigma}\bigr),
\end{split}
\]
using (\ref{eq:lower}). Analogously we obtain bounds for $P\in\co{V_{+,2}}$
and $P\in\co{V_{-,2}}$, giving rise to the inequalities
\begin{equation}
  \label{eq:bndPnullPoverQ}
  \sup_{P\in\co{V_{\cdot}}}\min\bigl\{P_0(p>0),P(p_0>0)\bigr\}
    \leq 1-f\bigl(\ft{\ep}{\sigma}\bigr),
\end{equation}
for $V_{\cdot}$ equal to $V_{+,1}$, $V_{-,1}$, $V_{+,2}$ and $V_{-,2}$.  
Combination of lemma~\ref{lem:HTBnd} and theorem~\ref{thm:main} now
shows that,
\[
  \Pi\bigl(\,g(\tht,\tht_0)<\ep\bigm|X_1,\ldots,X_n\,\bigr)
  \convas{P_0}1.
\]
The topology associated with the metric $g$ on $\Tht$
is equivalent to the restriction to $\Tht$ of the usual norm
topology on $\RR^2$, so that consistency with respect to the
pseudo-metric $g$ is equivalent to (\ref{eq:consdomain}).
\end{proof}

\subsection{Proofs for section~\ref{sec:rates}}

\begin{proof}{\it (theorem~\ref{thm:sieve})}\\
Fix $n\geq1$ large enough to satisfy conditions~{\it(i)}
and~{\it(ii)}.
According to lemma~\ref{lem:HTBnd}, there exist test functions
$\phi_{n,i}:\mathcal{X}^n\rightarrow[0,1]$ for all $1\leq i\leq N_n$,
such that, for all $\al\in[0,1]$,
\[
  P_0^n\phi_{n,i} + \sup_{P\in V_{n,i}}
    P_0^n\frac{dP^n}{d\Pi_n^\Pi}(1-\phi_{n,i})
    \leq \Pi(B_n)^{-\al} \pi_{P_0}(\co{V_{n,i}},B_n;\al)^n.
\]
Define
$\psi_n=\max_i\phi_{n,i}$ and decompose the $n$-th posterior for
$V_n=\{P\in\scrP:d(P,P_0)\geq\ep_n\}$, as follows,
\[
  \begin{split}
  P_0^n &\Pi(V_n|X_1,\ldots,X_n) \leq P_0^n\psi_n\\[1ex]
    & + P_0^n\Pi(V_n\cap\scrP_n|X_1,\ldots,X_n)(1-\psi_n)
    + P_0^n\Pi(\scrP\setminus\scrP_n|X_1,\ldots,X_n).
  \end{split}
\]
The first term is upper-bounded as follows,
\[
  P_0^n\psi_n\leq \sum_{i=1}^{N_n}P_0^n\phi_{n,i}
  \leq N_n \Pi(B_n)^{-\al} \pi_{P_0}(\co{V_{n,i}},B_n;\al)^n
  \leq e^{ (\frac{\al K}{2}- L)n\ep^2},
\]
for all $\al\in[0,1]$. The second term is bounded by,
\[
  \begin{split}
  P_0^n&\Pi(V_n\cap\scrP_n|X_1,\ldots,X_n)(1-\psi_n)
  \leq
  \sum_{i=1}^{N_n}P_0^n\Pi(V_{n,i}|X_1,\ldots,X_n)(1-\psi_n)\\
  &\leq
  \sum_{i=1}^{N_n}P_0^n\Pi(V_{n,i}|X_1,\ldots,X_n)(1-\phi_{n,i})
  \leq
  \sum_{i=1}^{N_n}\sup_{P\in V_{n,i}}
    P_0^n\frac{dP^n}{d\Pi_n^\Pi}(1-\phi_{n,i})\\[1ex]
  &\leq N_n \Pi(B_n)^{-\al} \pi_{P_0}(\co{V_{n,i}},B_n;\al)^n
  \leq e^{ (\frac{\al K}{2}- L)n\ep^2}\phantom{\Bigl(\Bigr)}
  \end{split}
\]
for all $\al\in[0,1]$. The third term requires condition~{\it(ii)},
\[
  \begin{split}
  P_0^n\Pi(&\scrP\setminus\scrP_n|X_1,\ldots,X_n)\\
    &= \int_{\scrP\setminus\scrP_n} P_0^n\Bigl(\frac{dP^n}{dP_n^\Pi}\Bigr)d\Pi(P)
    \leq \frac1{\Pi(B_n)}\int_{\scrP\setminus\scrP_n}
      P_0^n\Bigl(\frac{dP^n}{dP_n^{\Pi|B_n}}\Bigr)d\Pi(P)\\
    &\leq \frac{\Pi(\scrP\setminus\scrP_n)}{\Pi(B_n)}
      \sup_{P\in \scrP\setminus\scrP_n}\sup_{Q\in B_n}\bigl[P_0(dP/dQ)\bigr]^n
    \leq e^{-\frac{K}{2}n\ep_n^2}.
  \end{split}
\]
Choosing $0<\al<2L/K$, all three contributions go to zero as
$n\rightarrow\infty$.
\end{proof}

\begin{proof}{\it (lemma~\ref{lem:KLseparation})}\\
Assume that (\ref{eq:KLsepDelta}) holds. Like in the proof
of lemma~\ref{lem:KLequiv}, we have for all $\al\in(0,a)$,
\begin{equation}
  \label{eq:absconttwo}
  \sup_{Q\in B}\sup_{P\in W} P_0\Bigl(\frac{dP}{dQ}\Bigr)^\al
    \leq 1 + \al\,z(\al),
\end{equation}
where the function $z:[0,a)\rightarrow\RR$ is given by,
\[
  z(\al) = \sup_{Q\in B} \sup_{P\in W}
  P_0\Bigl(\frac{dP}{dQ}\Bigr)^{\al}\log\frac{dP}{dQ}.
\]
The function $z$ is convex and increasing, hence continuous
on $(0,a)$ and upper-semicontinuous at $a=0$ and maximal at $\al=a$.
Clearly, we have,
\[
  \lim_{\al\downarrow0}z(\al)\leq\sup_{Q\in B} \sup_{P\in W}
    P_0\log\frac{dP}{dQ} = \sup_{Q\in B}-P_0\log\frac{dQ}{dP_0}
    -\inf_{P\in W}-P_0\log\frac{dP}{dP_0},
\]
and the right-hand side is less than or equal to $-\Delta$.
By the continuity of $z$, there exists an $a'\in(0,a)$ such that,
$z(\al)\leq-\ft12\Delta$ for all $\al\in(0,a')$. Combining
(\ref{eq:absconttwo}) with the latter conclusion, we see that,
for all $\al\in(0,a')$,
\[
  \pi_{P_0}(B,W) = \inf_{0\leq\al\leq1}
    \sup_{Q\in B}\sup_{P\in W} P_0\Bigl(\frac{dP}{dQ}\Bigr)^\al
  \leq 1-\ft12 \al \Delta \leq e^{-\al''\Delta},
\]
where $\al''=\ft12\al$. Conversely, assume that (\ref{eq:KLsepDelta})
does not hold. Let $P\in W$, $Q\in B$ and $\al\in[0,1]$ be given;
by Jensen's inequality,
\[
  P_0\Bigl(\frac{dP}{dQ}\Bigr)^\al
    \geq \exp\Bigl(\al P_0\log\frac{dP}{dQ}\Bigr)
    = \exp\biggl( \al \Bigl( P_0\log\frac{dP}{dP_0}
        - P_0\log\frac{dQ}{dP_0}\Bigr)\biggr).
\]
Therefore, for all $\al\in(0,1)$,
\[
  \sup_{Q\in B}\sup_{P\in W}P_0\Bigl(\frac{dP}{dQ}\Bigr)^\al > e^{-\al\Delta}.
\]
\end{proof}

\begin{proof}{\it (corollary~\ref{cor:altGGV})}\\
Take $\scrP_n=\scrP$ for all $n\geq1$. Note that (\ref{eq:altGGV})
implies that $\Pi$ is a Kullback-Leibler prior, which implies that
$P_0^n\ll P_n^{\Pi}$, \cf\ corollary~\ref{cor:KLpriorpred}. Let
$V_n=\{P\in\scrP:H(P,P_0)\geq\ep_n\}$ and 
$B_n=\{\,P\in\scrP\,:\,-P_0\log(dP/dP_0)<\ep_n^2/8\,\}$. By
condition~{\it(i)} there is a cover of $V_n$ consisting of Hellinger
balls of radii $\ep_n/2$ of order $N_n=N(\ep_n,\scrP,H)\leq
\exp(L n\ep_n^2)$. Note that for every $1\leq i\leq N_n$ and all
$P\in \co{V_{n,i}}$, we have $-P_0\log(dP/dP_0)\geq H^2(P,P_0)\geq
(H(V_n,P_0)-\ep_n/2)^2=\ep_n^2/4$, while
$-P_0\log(dQ/dP_0)\leq\ep_n^2/8$ for all $Q\in B_n$. According to
lemma~\ref{lem:KLseparation}, the separation in Kullback-Leibler
divergence between $B_n$ and $V_n$ implies that
$\pi_{P_0}\bigl(\co{V_{n,i}},B_n\bigr)\leq e^{-\al \ep_n^2}$ for
some $\al>0$. Possibly after rescaling of $\ep_n$ by an
$n$-independent constant (which leads to larger $\al$, effectively),
$\pi_{P_0}$ satisfies condition (\ref{eq:pitwo}). The assertion then
follows from theorem~\ref{thm:sieve}.
\end{proof}


\bibliographystyle{imsart-nameyear}

\end{document}